\documentclass[11pt, notitlepage]{article}
\usepackage{amssymb,amsmath,comment}
\catcode`\@=11 \@addtoreset{equation}{section}
\def\thesection{\arabic{section}}

\def\theequation{\thesection.\arabic{equation}}
\catcode`\@=12
\usepackage{enumitem}
\usepackage{colortbl}%
\usepackage{a4wide}
\usepackage{parskip}
\usepackage{hyperref}
\usepackage{cite}

\newcommand{\ds} {\displaystyle}
\newcommand{\e}{\epsilon}

\newcommand{\al} {\alpha}
\newcommand{\ba} {\beta}

\newcommand{\de} {\delta}

\newcommand{\Om} {\Omega}
\newcommand{\ra} {\rightarrow}

\newcommand{\De} {\Delta}
\newcommand{\la} {\lambda}
\newcommand{\La} {\Lambda}
\newcommand{\noi} {\noindent}
\newcommand{\na} {\nabla}

\newcommand{\mb} {\mathbb}
\newcommand{\mc} {\mathcal}

\newcommand{\ld} {\langle}
\newcommand{\rd} {\rangle}
\newcommand{\I}{\int\limits_}

\usepackage[all]{xy}
\catcode`\@=11
\def\theequation{\@arabic{\c@section}.\@arabic{\c@equation}}
\catcode`\@=12

\def\QED{\hfill {$\square$}\goodbreak \medskip}
\newtheorem{Theorem}{Theorem}[section]
\newtheorem{Lemma}[Theorem]{Lemma}
\newtheorem{Proposition}[Theorem]{Proposition}
\newtheorem{Corollary}[Theorem]{Corollary}
\newtheorem{Remark}[Theorem]{Remark}
\newtheorem{Definition}[Theorem]{Definition}

\def\XXint#1#2#3{{\setbox0=\hbox{$#1{#2#3}{\int}$ }
		\vcenter{\hbox{$#2#3$ }}\kern-.6\wd0}}

\begin{document}
	{\vspace{0.01in}
		\title
		{Mixed local and nonlocal elliptic equation with singular and critical Choquard nonlinearity}
		\author{ {\bf G.C. Anthal\footnote{	Department of Mathematics, Indian Institute of Technology, Delhi, Hauz Khas, New Delhi-110016, India. e-mail: Gurdevanthal92@gmail.com}\;, J. Giacomoni\footnote{ 	LMAP (UMR E2S UPPA CNRS 5142) Bat. IPRA, Avenue de l'Universit\'{e}, 64013 Pau, France. 
					email: jacques.giacomoni@univ-pau.fr}  \;and  K. Sreenadh\footnote{Department of Mathematics, Indian Institute of Technology, Delhi, Hauz Khas, New Delhi-110016, India.	e-mail: sreenadh@maths.iitd.ac.in}}}
		\date{}
		
		\maketitle
		
		\begin{abstract}
		In this article, we study a class of elliptic problems involving both local and nonlocal operators with different orders and a singular nonlinearity in combination with critical Hartree type nonlinearity (see problem $(P_\la)$ below). Using variational methods together with the critical point theory of nonsmooth analysis, we show the existence, the regularity and the multiplicity of weak solutions with respect to the parameter $\la$.\\
		\noi \textbf{Key words:} Local-nonlocal mixed operators, singular nonlinearity,  Choquard equation, nonsmooth analysis, existence results, Sobolev regularity, comparison principle, asymptotic behaviour. \\
		
		\noi \textit{2020 Mathematics Subject Classification:} 35A01, 35A15, 35B33, 35D30, 35J15, 35J75, 35M10, 49J52.\\
		\end{abstract}	
	\section{Introduction}
	This article investigates the existence and the multiplicity of weak solutions to the following problem:
	\begin{equation*} 
	( P_\la)	\begin{cases}
			 \mc M u= u^{-\gamma}+\la \left(\ds\I\Om \frac{|u|^{2_\mu^\ast}(y)}{|x-y|^\mu}dy\right)|u|^{2_\mu^\ast-2}u,~&u>0~\text{in}~\Om,\\
			 u=0 &\text{in}~\mb R^n\setminus \Om,
		\end{cases}
	\end{equation*}
where $\gamma >0$, $n \geq 3$, $s \in (0,1)$, $0<\mu<n$, $2_\mu^\ast =\frac{2n-\mu}{n-2}$ and $\Om$ is a bounded domain in $\mb R^n$ with smooth boundary. The operator $\mc M$ in $( P_\la)$ is defined by 
\begin{equation}\label{ecl}
	\mc M =-\De +(-\De)^s~\text{for some}~ s \in (0,1),
\end{equation}
 i.e., composed by a local operator,  Laplacian $(-\De)$, and a nonlocal operator, the fractional Laplacian $(-\De)^s$,  given for a fixed parameter $s \in (0,1)$,  by 
\begin{align*}
	(-\De)^su =C(n,s)P.V. \I{\mb R^n} \frac{u(x)-u(y)}{|x-y|^{n+2s}}dy.
\end{align*}
The term $``P.V"$ stands for the Cauchy's principal value and $C(n,s)$ is a normalizing constant, whose explicit expression is given by
\begin{equation*}
	C(n,s)=\left(\I{\mb R^n}\frac{1-\cos(z_1)}{|z|^{n+2s}}dz\right)^{-1}.
\end{equation*}
 The study of mixed type operators of the form $\mc M$ in \eqref{ecl} is motivated by a wide range of applications. Indeed these operators arise naturally in the applied sciences, to study the impact caused by a local and a nonlocal change in a physical phenomenon. In particular, these operators model diffusion patterns with different time scales (roughly speaking, the higher order operator leads the diffusion for small scales times and the lower order operator becomes predominant for large times). They appear for instance in bi-modal power law distribution processes, see \cite{PV1}. Further applications occur in the theory of optimal searching, biomathematics and animal forging, see \cite{DLV, DV} and the references therein. 

\noindent Due to these various important applications, the study of elliptic problems involving mixed type operators having both local and nonlocal features has attracted more and more attention. In particular, the current research work on problems involving this type of operators has investigated several issues about existence and regularity of solutions. In this mater, using probability theory, Foondum \cite{F}, Chen et. al \cite{CKSV}, studied the regularity of solutions to the equation
\begin{equation*}
	\mc M u =0.
\end{equation*}
More recently, using a purely analytic approach, Biagi, Dipierro, Valdonoci and Veechi, in their series of papers \cite{BDVV1,BDVV2,BDVV4}, have carried out a broad investigation of problems involving mixed operators, proving a number of results, concerning regularity and qualitative behaviour of solutions, maximum principles and related variational principles. The question of H\"{o}lder regularity was investigated by De Filippis-Mingione in \cite{FM} for a large class of mixed local and nonlocal  operators. Under some suitable assumptions, the authors proved almost Lipschitz continuity and local gradient H\"{o}lder regularity (see Theorem 3, 6 and 7 respectively there).\\
	There is a large literature available for the problems with Choquard nonlinearity due to its vast application in physical modeling, see for instance the works of Pekar \cite{Pe} and Lieb \cite{Li}. For detailed studies on the existence and regularity of weak solutions for these types of problems we refer in the local setting to \cite{MS4} and the references therein. In the non local case, Choquard type equations have been investigated more recently and arise for instance in the study of mean field limit of weakly interacting molecules, in the quantum mechanical theory and in the dynamics of relativistic Boson-stars (see \cite{daveniaetal} and references therein). In \cite{daveniaetal}, a Schr\"odinger type problem involving a Hartree type nonlinearity and the fractional laplacian is studied. Existence, nonexistence and properties of solutions are proved in this paper.\\
For the Brezis-Niremberg type problems  involving Choquard nonlinearities, we refer to \cite{GY} in the local setting,  to \cite{MS} for fractional diffusion case and to \cite{AGS} for the mixed operator case.\\
Problems involving singular nonlinearities have a very long history. One of the seminal breakthrough in the study of singular elliptic problems was the work of Crandall, Rabinowitz and Tartar \cite{CRT}. In this work, the authors proved the existence of  solutions to a class of  elliptic equations involving a singular nonlinearity by using the classical method of sub and supersolutions and  passing to the limit on approximated problems. Subsequently, a large amount of works have been done discussing further issues about regularity, mutilplicity and qualitative properties of solutions to quasilinear elliptic and parabolic singular problems, see for instance the review articles \cite{GR, HM} and references therein.\\
 Without giving an exhaustive list of contributions, we quote some key references: \cite[Haitao]{H} brought new multiplicity results of solutions to an elliptic singular problem with critical perturbation. Precisely, the author considered the following problem:
 \begin{align*}
 	-\De u =\la u^{-\gamma} + u^p,~u >0~\text{in}~\Om,~u =0~\text{on}~\partial\Om,
  \end{align*}
where $\Om \subset \mb R^n~(n \geq 3)$ is a smooth bounded domain, $\gamma \in (0,1)$, and  $1<p\leq \frac{n+2}{n-2}$. Using monotone iterations and the mountain pass lemma, the author showed existence and multiplicity results for the maximal range of parameter $\la$, i.e. established  global multiplicity. We also refer to \cite{AG, DPST} for higher singular cases, i.e. with $\gamma \in (1,3)$. Finally, the case of any $\gamma >0$ was considered by Hirano, Saccon and Shioji in \cite{HSS}. Here the authors studied the existence of very weak solutions $u\in H^1_{loc}(\Om)$ satisfying $(u-\e)^+ \in H^1_0(\Om)$ for all $\e >0$. The proof used variational methods and nonsmooth analysis arguments.\\
Investigating a class of nonlocal elliptic and singular equations, Barrios et al. \cite{BDMP} considered the following type of problem:
\begin{align*}
	(-\De)^s u =\la \frac{g(x)}{u^\gamma} +K u^r,~u>0~\text{in}~\Om, ~u =0 ~\text{in}~\mb R^n \setminus \Om,
\end{align*}
where $n >2s$, $ K \geq 0$, $0<s<1$, $\gamma >0$, $\la>0$, $1 \leq r <2_s^*-1$ with $2_s^* =\frac{2n}{n-2s}$ and $g \in L^p(\Om)$, with $p \geq 1$ a nonnegative function. In the spirit of \cite{CRT}, the authors considered first the approximated problem where the singular term $\frac{1}{u^\gamma}$ is replaced by $\frac{1}{(u+1/k)^\gamma}$ and showed the existence of a solution $u_k$. Finally, the existence of weak solutions to the initial problem is obtained from uniform estimates of $\{u_k\}_{k\in\mathbb{N}}$. Furthermore the authors proved multiplicity results when $K>0$ and for small $\la>0$. The case of critical exponent problem with singular nonlinearity was handled by \cite{MS1} for $\gamma \in (0,1)$. Later in the spirit of \cite{HSS}, using the nonsmooth analysis, \cite{GMS} established the multiplicity result for the critical exponent problem with singular nonlinearity for any $\gamma >0$. More recently, \cite{GGS1} studied the following nonlocal singular problem with a critical Choquard type nonlinearity:
\begin{align*}
	(-\De)^s u=  u^{-\gamma}+\la \left(\ds\I\Om \frac{|u|^{2_{\mu,s}^\ast}(y)}{|x-y|^\mu}dy\right)|u|^{2_{\mu,s}^\ast-2}u,~u>0~\text{in}~\Om,
	u=0~ \text{in}~\mb R^n\setminus \Om,
\end{align*}
wher $\gamma>0$, $n >2s$, $2_{\mu,s}^* =\frac{2n-\mu}{n-2s}$, and $\Om$ a bounded domain in $\mb R^n $ with smooth boundary.
Again using the critical point theory of nonsmooth analysis and the geometry of the energy functional, the authors established a global multiplicity result. 
\\
In the current literature, singular problems involving both local and nonlocal operators is very less investigated. Recently, \cite[Arora and Radulescu]{AR} studied the following singular problem involving mixed operators:
\begin{align*}
	\mc M u = \frac{g(x)}{u^\gamma},~u >0~\text{in}~\Om,~ u=0~\text{on}~\mb R^n\setminus \Om,
\end{align*} 
where $\Om \subset \mb R^
n$, $n \geq 2$, $\gamma \geq 0$, and $g: \Om \ra  \mb R^+$ belongs to
$L^r(\Om)$ for some $1 \leq r \leq \infty$. The case where $g$ behaves as a power function of distance function $\de$ near the boundary i.e.
$g(x)\sim\de^{-\zeta}(x)$ for some $\zeta \geq 0$ and $x$ lying near the boundary $\partial \Om$ is also considered in \cite{AR}. Here, the authors proved  the existence, Sobolev regularity and boundary behaviour of weak solutions under different assumptions  on $g$ and $\gamma$. We also refer to \cite{G} for the proof of
the existence of multiple solutions in case of perturbed subcritical singular nonlinearities with $\gamma \in (0,1)$. The case of  quasilinear mixed operators  is issued in \cite{GU}, still with the restriction $\gamma \in (0,1)$.\\
Motivated by the above discussion, in the present paper, we consider the singular doubly nonlocal problem $(P_\la)$ with any $\gamma>0$. To the best of our knowledge, there is no previous contribution dealing with  critical Choquard type problems involving singular nonlinearities and mixed diffusion operators. In the spirit of \cite{HSS}, using the theory of nonsmooth analysis together with convexity properties of the associated energy functional, we prove the existence, multiplicity (under some additional restrictions on $s$ and $n$) asymptotic behavior and regularity of weak solutions to $(P_\la)$ for all $\gamma>0$. Precisely, we prove the following main result :
	\begin{Theorem}\label{tmr}
		Let $\mu < \min\{4,n\}$ and $\mc G$ be defined as in Theorem \ref{ART}. Then the following assertions hold:
		\begin{itemize}
		\item[(i)](Existence/nonexistence) 
		There exists $\Lambda>0$ such that $(P_\la)$ admits at least one weak solution for every $\la \in (0,\Lambda]$ and no solution for $\la > \Lambda$. 
		\item[(ii)](Asymptotic behavior and regularity) Any weak solution $u$ to  $(P_\la)$  is bounded, satisfies  $u \in \mc G(\Om)\cap C^{0,\al}_{\text{loc}}(\Om)$, for some $\al \in (0,1)$, and for any $\nu>\max\{1,\frac{\gamma+1}{4}\}$, $u^\nu\in H^1_0(\Om)$.  Furthermore, if $\gamma<3$, then $u\in H^1_0(\Om)\cap C^{1,\beta}_{loc}(\Om)$ for some $0<\beta<1$.
		\item[(iii)](Multiplicity) Assuming $n+2s <6$, there exists at least two distinct weak solutions for any $\lambda\in (0,\Lambda)$. 
		\end{itemize}
	\end{Theorem}
\begin{Remark} \label{remark1.2}
With additional restrictions, we can further show the following regularity results (proved at the end of Section 3):
\begin{enumerate}[label=(\arabic*)]
	\item Suppose $\gamma \in (0,\frac{1}{n})$ and $s \in (0,\frac12)$. Then infact $u \in C^{2,\al-2s}_{\text{loc}}(\Om)$ with any $\alpha<1$.
	\item If $\gamma \in (0,1)$ and $s \in (0,\frac{1}{2})$, then we have $u \in W^{2,p}(\Om)$ with $p \in (1,\frac{1}{\gamma})$.
\end{enumerate}
	\end{Remark}
 An important step to prove Theorem \ref{tmr} is to prescribe  the asymptotic behaviour of weak solutions to  $(P_\la)$ near the boundary $\partial\Omega$. To this goal, we prove a crucial comparison principle (see Lemma \ref{lcp}) in the spirit of \cite{GGS1}. In addition, we show that weak solutions  are bounded by using Lemma \ref{lcp} and bootstrap techniques applied on the equivalent problem $(\hat{P}_\la)$ obtained by translating  $(P_\la)$ with the solution to the pure singular problem  $(P_0)$, defined in section 2. Next we establish the existence of  energy solutions to $(\hat{P}_\la)$ trapped in a conical shell  defined by sub and supersolutions with suitable asymptotic behaviour near boundary. From this,  regularity and asymptotic behaviour of weak solutions follow.  Finally, using variational arguments, we prove the global multiplicity result of problem thanks to convex properties of the singular part of the nonlinearity and accurate estimates  about energy levels associated to $(\hat{P}_\la)$. We want to highlight that our approach can be apply also in other situations, in particular when considering local critical perturbation instead of critical Choquard nonlinearity. The corresponding  results are new for large $\gamma$  according to former contributions.\\ 
\textbf{Organization of the paper}: In Section \ref{PMR} we define the function spaces, give some preliminaries of nonsmooth analysis and state technical results used in the subsequent sections. In Section \ref{s3}, we prove the boundedness of weak solutions to $(P_\la)$ by bootstrap type arguments together with the comparison principle proved in Lemma \ref{lcp}. Finally, in Section \ref{s4} we show the existence of weak solutions (by sub and supersolutions technique) and derive their regularity and asymptotic behaviour. Identifying the first critical energy level under which the Palais Smale condition holds, we prove the multiplicity of weak solutions to $(P_\la)$ and complete the proof of Theorem \ref{tmr}.\\
\textbf{Notations}: Throughout the paper, we set
\begin{itemize}
	\item $\de(x):=\text{dist} (x,\partial \Om)$ and $d_\Om =\text{diam}(\Om)$;
	\item for any number $p \in (1,\infty)$, we denote by $p^\prime =\frac{p}{p-1}$ as the conjugate exponent of $p$ and $|\cdot|_p$ denotes the norm in $L^p(\mb R^n)$ space;
	\item for any two functions $g,~h$, we write $g \prec h$ or $g \succ h$ if there exists a constant $C>0$ such that $g \leq Ch$ or $g \geq Ch$. We write $g \sim h$ if $g \prec h$ and $ g \succ h$; 
	\item $u^{p} =|u|^{p-1}u$ and $\|u\|^{22_\mu^\ast}_{HL}=\ds\I{\mb R^n}\I{\mb R^n}\frac{|u|^{2_\mu^\ast}(x)|u|^{2_\mu^\ast}(y)}{|x-y|^\mu}dxdy
	$.
\end{itemize}  
	\section{Preliminaries and auxilary results}\label{PMR}
		In this section we give the functional settings and collect the notations and preliminary results required in the rest of the paper.\\
	Let $s \in (0,1)$. For a measurable function $u:\mb R^n \ra \mb R$, we define
	\begin{equation*}
		[u]_s =\left(\frac{C(n,s)}{2} \I{\mb R^n}\I{\mb R^n}\frac{|u(x)-u(y)|^2}{|x-y|^{n +2s}}dxdy\right)^\frac12,
	\end{equation*}
	the so-called Gagliardo seminorm of $u$ of order $s$.	\\
 We define the space $X_0$ as the completion of $C_c^\infty (\Om)$ with respect to the global norm
	\begin{align*}
		\|u\|:= \left(  \ll u\gg^2+[u]_s^2\right)^\frac12, ~ u \in C_c^\infty(\Om),
	\end{align*}
	where we define
	 	$ \ds \ll u\gg =\left(\ds\I{\mb R^n}|\na u|^2\right)^\frac12.$
  The norm $\|\cdot\|$ is induced by the scalar product
	\begin{equation*}
		\ld u,v\rd:= \I{\mb R^n} \na u\na v dx + \frac{C(n,s)}{2}\I{\mb R^n}\I{\mb R^n}\frac{(u(x)-u(y))(v(x)-v(y))}{|x-y|^{n+2s}}dxdy,~u,v \in X_0.
	\end{equation*}
	 Clearly, $X_0$ is a Hilbert space.
	\begin{Remark}
		 Note that in the definition of $\|\cdot\|$ the $L^2$-norm of $\na u$ is considered on the whole of $\mb R^n$ in spite of $u \in C_c^\infty(\Om)$ (identically vanishes outside $\Om$). This is to point out that the elements in $X_0$ are functions defined on the entire space and not only on $\Om$. The benefit of having this global functional setting is that these functions can be globally approximated on $\mb R^n$ with respect to the norm $\| \cdot\|$ by smooth functions with support in $\Om$.\\
		We see that this global definition of $\|\cdot\|$ implies that the functions in $X_0$ naturally satisfy the nonlocal Dirichlet type condition prescribed in problem $( P_\la)$, that is,
		\begin{equation}\label{e1.2}
			u \equiv 0 ~\text{a.e. in}~\mb R^n\setminus\Om~\text{for every}~u \in X_0.
		\end{equation} 
	In order to verify \eqref{e1.2}, we know (see \cite[Proposition 2.2]{DPV}) that $H^1(\mb R^n)$ is continuously embedded into $H^s(\mb R^n)$ (with $s\in(0,1)$) $i.e.$ there exists a constant $k =k(s)>0$ such that, for every $u \in C_c^\infty(\Om)$ one has
	\begin{equation*}
		[u]_s^2 \leq k(s)\|u\|_{H^1(\mb R^n)}^2=k(s)(\|u\|_{L^2(\mb R^n)}^2+\ll u\gg^2).
	\end{equation*}
	This, together with the classical Poincar\'{e} inequality, implies that $\|\cdot\|$ and the full $H^1-norm$ in $\mb R^n$ are actually equivalent in the space $C^\infty_c(\Om)$, and hence
	\begin{align*}
		X_0=\overline{C_c^\infty(\Om)}^{\|\cdot\|_{H^1(\mb R^{n})}}=\{u \in H^1(\mb R^n): u\arrowvert_\Om\in H_0^1(\Om)~\text{and}~u\equiv 0 ~\text{a.e. in}~\mb R^n\setminus \Om\}.
	\end{align*}
	\end{Remark}
Now we recall the Hardy-Littlewood-Sobolev inequality which is the first brick in study of the Choquard type problems.
\begin{Proposition}\label{hls}
	\textbf{Hardy-Littlewood-Sobolev inequality} Let $r, q>1$ and $0<\mu<n$ with $1/r+1/q+\mu/n=2$, $g\in L^{r}(\mb R^n), h\in L^q(\mb R^n)$. Then, there exists a sharp constant $C(r,q,n,\mu)$ independent of $g$ and $h$ such  that 
	\begin{equation*}
		\int\limits_{\mb R^n}\int\limits_{\mb R^n}\frac{g(x)h(y)}{|x-y|^{\mu}}dxdy \leq C(r,q,n,\mu) |g|_r|h|_q.
	\end{equation*}                             
\end{Proposition}
In particular, let $g = h = |u|^p$ 
then by Hardy-Littlewood-Sobolev inequality we see that,
$$                               \int\limits_{\mb R^n}\int\limits_{\mb R^n}\frac{|u(x)|^pu(y)|^p}{|x-y|^{\mu}}dxdy$$
is well defined if $|u|^p \in L^\nu(\mb R^n)$ with $\nu =\frac{2n}{2n-\mu}>1$.
Thus, from Sobolev embedding theorems, we must have
\begin{equation*}
\frac{2n-\mu}{n} \leq p \leq \frac{2n-\mu}{n-2}.
\end{equation*}
From this, for $u \in H^1(\mb R^n) $ we have
$$    \left(\int\limits_{\mb R^n}\int\limits_{\mb R^n}\frac{|u(x)|^{2_\mu^\ast}|u(y)|^{2_\mu^\ast}}{|x-y|^\mu}dxdy \right)^\frac{1}{2_\mu^\ast} \leq C(n,\mu)^\frac{1}{2_\mu^\ast}  |u|_{2^*}^2  .                                     $$
	We denote by $S_{H,L}$ the best constant associated to Hardy-Littlewood-Sobolev inequality, i.e,
\begin{align*}
		S_{H,L}=\inf\limits_{u \in C_0^\infty(\mb R^n)\setminus \{0\}} \frac{\|\nabla u\|_{L^2(\mb R^n)}^2}{\|u\|_{HL}^2}.
\end{align*}
 The best constant $S_M$ of the mixed Sobolev embedding is defined by
\begin{equation*}
	S_M = \inf_{u \in X_0\setminus\{0\}}\frac{\|u\|^2}{|u|^2_{2^*}}.
\end{equation*}
We also define '
\begin{equation*}
	S_{H,L,M} = \inf_{u \in X_0\setminus\{0\}}\frac{\|u\|^2}{\|u\|^2_{HL}}.
\end{equation*}
 From \cite[Theorem 1.1]{BDVV} and \cite[Theorem 1.2]{AGS} one has that $S_M =S$ and $S_{H,L,M}=S_{H,L}$, where $S$ is the best constant in the classical Sobolev embedding. Next, the following lemma plays a crucial role in the sequel:
\begin{Lemma}\cite{GY}\label{l2.2}
	The constant $S_{H,L}$ is achieved if and only if 
	\begin{equation*}
		u=C\left(\frac{b}{b^2+|x-a|^2}\right)^\frac{n-2}{2},
	\end{equation*}
where $C>0$ is a fixed constant, $a \in \mb R^n$ and $b \in (0,\infty)$ are parameters. Moreover,
\begin{equation*}
	S= C(n,\mu)^\frac{n-2}{2n-\mu}S_{H,L}.
\end{equation*}
\end{Lemma}
Now we give the notion of a weak solution to problem $( P_\la)$.
\begin{Definition}\label{d2.3}
	We say that a function $u$ is a weak solution of $( P_\la)$ if the following assertions hold:
	\begin{enumerate}[label=(\roman*)]
		\item $u^\ell \in X_0$ for some $\ell \geq1$.
		\item $\inf_{x\in K}u(x)\geq m_K$ , with $m_K >0$, for every compact subset $K \subset \Om$.
		\item For any $\psi \in C^\infty(\mb R^n)$ with compact support in $\Om$,
		\begin{equation}\label{ewsd}
			\ld u,\psi\rd =\I\Om u^{-\gamma}\psi dx +\la \I\Om\I\Om\frac{u^{2_\mu^\ast}(y)u^{2_\mu^\ast-1}(x)\psi(x)}{|x-y|^\mu}dxdy.
		\end{equation}
	\end{enumerate} 
\end{Definition}
\begin{Remark}Some remarks are in order.
\begin{enumerate}
	\item  From (i)-(ii) of definition \ref{d2.3}, we can easily check that a weak solution $u$  satisfies $u\in H^1_{\text{loc}}(\Om)$ and $(u-\e)^+ \in X_0$ for every $\e>0$. 
	\item We also pointout that \eqref{ewsd} is well defined, i.e. all the integrals in \eqref{ewsd} are finite. Indeed, if $ \varphi \in C_0^\infty(\Om)$ with $K =\text{supp} \varphi$, we have
	\begin{align*}
		\left|\I{\mb R^n}  \na u \na \varphi dx\right| 
		\leq& | \na u|_{L^2(K)} | \na \varphi|_{L^2(K)} <\infty.
	\end{align*}
Let us set
\begin{align*}
	\mc S_\varphi:= \text{supp}\varphi,~\mc Q_\varphi =\mc R^{2n}\setminus (\mc S_\varphi^c \times\mc S_\varphi^c).
\end{align*}
Now by \cite[Lemma 3.5]{CMSS}, with $q =\ell$, we have that
\begin{align*}
	|u^\ell(x)-u^\ell(y)|^2 \geq m_K^{2(\ell-1)}|u(x)-u(y)|^2~\text{in}~\mc Q_\varphi,
\end{align*}
since either $u(x) \geq m_K$ in $\mc Q_\varphi$ or $u(y) \geq m_K$ in $\mc Q_\varphi$. From this we easily infer that
\begin{align*}
&\left |\I{\mb R^n}\I{\mb R^n} \frac{(u(x)-u(y))(\varphi(x)-\varphi(y))}{|x-y|^{n+2s}}dxdy \right |\\=& \left|\I{\mc Q_\varphi}\frac{(u(x)-u(y))(\varphi(x)-\varphi(y))}{|x-y|^{n+2s}}dxdy\right|\\
	\leq & \left(\I{\mc Q_\varphi}\frac{(u(x)-u(y))^2}{|x-y|^{n+2s}}dxdy\right)^\frac12\left(\I{\mc Q_\varphi}\frac{(\varphi(x)-\varphi(y))^2}{|x-y|^{n+2s}}dxdy\right)^\frac12\\
	\leq & Mm_K^{2(1-\ell)} \left(\I{\mb R^n}\I{\mb R^n}\frac{(u^\ell(x)-u^\ell(y))^2}{|x-y|^{n+2s}}dxdy\right)^\frac12 <\infty.
\end{align*}
	Concerning the right hand side of \eqref{ewsd}, since $ u \in L^{2^*}(\Om)$ using classical Hardy-Littlewood-Sobolev and H\"{o}lder's inequality one has
	\begin{align*}
	\left|	\I{\Om}\I{\Om}\frac{|u(x)|^{2_\mu^*}|u(y)|^{2_\mu^*-2}u(y)\varphi(y)}{|x-y|^\mu}dxdy\right|
		\leq C(n,\mu)|u|_{2^*}^2|\varphi|_{2^*}<\infty. 
	\end{align*}
Also since $\varphi \in C_c^\infty(\Om)$,
\begin{align*}
	\I\Om u^{-\gamma}\varphi dx \leq \frac{1}{m_K^\gamma}\I\Om \varphi dx <\infty.
\end{align*}
\end{enumerate}
\end{Remark}
\begin{Lemma}\label{l2.5}
	Let $u$ be a weak solution to $( P_\la)$. Then for all compactly supported $0\leq \psi \in X_0\cap L^\infty(\Om)$, we have  
	\begin{equation*}
		\ld u,\psi\rd=\I\Om u^{-\gamma}\psi dx+\la \I\Om\I\Om\frac{u^{2_\mu^\ast}(y)u^{2_\mu^\ast-1}(x)\psi(x)}{|x-y|^\mu}dxdy.
	\end{equation*}
\end{Lemma}
\begin{proof}
	Proof follows on the similar lines of proof of \cite[Lemma 2.9]{GGS1}. \QED
	\end{proof}
In order to prove the existence results for $( P_\la)$, we translate the problem by the unique solution to the purely singular problem:
\begin{equation*}
( P_0)\begin{cases}\mc M u =u^{-\gamma},~u>0~\text{in}~\Om,~u=0~\text{in}~\mb R^n\setminus \Om.
	\end{cases}
\end{equation*}
Using Theorems 2.6, 2.7, and 2.8 of \cite{AR}, we have the following:
\begin{Theorem}\label{ART}  We have the following
	\begin{enumerate}[label=(\roman*)]
		\item 	There exists a positive minimal solution $ \hat{u} \in H^1_{loc}(\Om) \cap L^\infty(\Om)$ of $(P_0)$ such that
		for every $ \psi \in H^1(\mb   R^n)$ with compact support in $\Om$ we have
		\begin{align*}
			\ld \hat{u},\psi\rd =\I\Om \hat{u}^{-\gamma} \psi dx.
		\end{align*}	
			\item $\inf_{x\in K}\hat{u}(x)>0$ for every compact subset $K \subset \Om$.
\item For any $\gamma >0$, we have the following asymptotic behavior:
\begin{align*}
	\hat{u} \in \mc G (\Om)~\text{where}~\mc G(\Om)=\begin{cases} 
		u : u \sim \de~&\text{if}~\gamma<1,\\
		u: u\sim \de \ln^\frac12\left(\frac{d_\Om}{\de}\right)~&\text{if}~\gamma=1,\\
		u: u\sim \de^\frac{2}{\gamma+1}~&\text{if}~\gamma >1,
		\end{cases}
\end{align*}
and the Sobolev regularity
\begin{align*}
	\hat{u}^\nu \in H^1_0(\Om)~\text{with}~\nu \begin{cases}
	=1~&\text{if}~\gamma <3,\\
>	\frac{\gamma+1}{4}~&\text{if}~\gamma \geq 3.
\end{cases}
\end{align*}
\end{enumerate}
\end{Theorem}
 Now we consider the following translated problem:
\begin{equation*}
(\hat{P}_\la)\begin{cases}	\mc M u+\hat{u}^{-\gamma}-(u+\hat{u})^{-\gamma}=\la\left(\ds\I\Om\frac{(u+\hat{u})^{2_\mu^\ast}}{|x-y|^\mu}dy\right)(u+\hat{u})^{2_\mu^\ast-1},~u>0~\text{in}~\Om,\\
	u=0~\text{in}~\mb R^n\setminus\Om.
	\end{cases}
	\end{equation*}
 Define the function
$f:\Om\times\mb R \rightarrow\mb R\cup\{-\infty\} $ by
\begin{equation*}
	f(x,\tau)=\begin{cases}
		(\hat{u}(x))^{-\gamma}-(\tau+\hat{u}(x))^{-\gamma}~&\text{if}~\tau+\hat{u}(x)>0,\\
		-\infty &\text{otherwise}.
	\end{cases}
\end{equation*}
 Also we define $\ds F(x,\tau)=\I0^\tau f(x,r)dr$. Note that $F$ is nonnegative and nondecreasing in $\tau$. Next we define the notion of subsolution and supersolution for problem $(\hat{P}_\la)$.

\begin{Definition}
	A nonnegative function $ u \in X_0$ is a subsolution (resp. a supersolution) of $(\hat{P}_\la)$ if the following hold:
	\begin{enumerate}[label=(\roman*)]
		\item $f(\cdot,u) \in L^1_{\text{loc}}(\Om)$;
		\item For any nonnegative $\psi \in C_c^\infty(\Om)$,
		\begin{equation*}
			\ld u,\psi\rd +\I\Om f(x,u)\psi dx -\la \I\Om\I\Om\frac{(u+\hat{u})^{2_\mu^\ast}(y)(u+\hat{u})^{2_\mu^\ast-1}(x)\psi(x)}{|x-y|^\mu}dxdy\leq 0~(\text{resp.}\geq 0).
		\end{equation*}
	\end{enumerate}
\end{Definition}
\begin{Definition}
	A function $u \in X_0$ is a weak solution to $(\hat{P}_\la)$ if it is both a subsolution  and a supersolution.
\end{Definition}
Our first lemma is the following implication.
\begin{Lemma}\label{exi}
	Suppose $u \in X_0$ is a weak solution of $(\hat{P}_\la)$. Then $u +\hat{u}$ is a weak solution of $(P_\la)$.
\end{Lemma}
\begin{proof}
	Since $\hat{u}$ satisfies $(ii)$ of Definition \ref{d2.3} and $u$ is nonnegative, $u +\hat{u}$ also satisfies it. Now, $u +\hat{u}$ clearly satisfies \eqref{ewsd} provided $(u+\hat{u})^\ell \in X_0$ for some $\ell \geq 1$. We consider the following two cases:\\
	\textbf{Case A:} $\gamma \in (0,3)$. In this case, using Theorem \ref{ART}, we have $\hat{u} \in H_0^1(\Om)$,  $\hat{u}=0$ in $\mb R^n\setminus \Om$, and so $\hat{u} \in H^1(\mb R^n)$. Thus, clearly $u+\hat{u} \in X_0$.\\
	\textbf{Case B}: $\gamma \in [3,\infty)$. In this case, again using Theorem \ref{ART}, we see that $\hat{u}^\ell \in H_0^1(\Om)$ with $\ell =\frac{\gamma+1}{4} >1$. We claim that $(u+\hat{u})^\ell \in X_0$. Since $u+\hat{u} \in H^1_{\text{loc}}(\Om) \cap L^\infty(\mb R^n)$ (see Lemma \ref{lib}) and $u +\hat{u} =0$ a.e. in $\mb R^n \setminus \Om$, to prove the claim it is sufficient to show that $\na ( u+\hat{u})^\ell \in L^2(\Om)$, i.e.,
	\begin{align}\label{equation2.6}
		|\na (u+\hat{u})^\ell|^2 =\ell^2 |\na (u+\hat{u})(u+\hat{u})^{\ell-1}|^2 \in L^1(\Om).
	\end{align}
	Since by Lemma \ref{l3.6}, $u +\hat{u}$ behaves as $\hat{u}$ near boundary, to prove \eqref{equation2.6}, it suffices to show that $|\hat{u}^{\ell-1}\na \hat{u}|^2 \in L^1(\Om)$, which is true as $\hat{u}^\ell \in H_0^1(\Om)$. \QED
	\end{proof}
\begin{Lemma}\label{l2.6}
	For each $v \in X_0$, $v\geq0$, there exists a sequence $\{v_k\}_{k\in \mb N} \subset X_0$ such that $v_k \ra v$ strongly in $X_0$, where $0\leq v_1\leq v_2 \leq \cdots$ and $v_k$ has compact support in $\Om$ for each $k$.
\end{Lemma}
\begin{proof}
The proof is similar to the proof of \cite[Lemma 3.1]{GMS} and hence omitted.
 \QED
	\end{proof}
\begin{Lemma}\label{l2.7}
	Let $u \in X_0$ be a weak solution to $(\hat{P}_\la)$. Then for any $\psi \in X_0$, we have
	\begin{equation}\label{e2.4}
		\ld u,\psi\rd+\I\Om f(x,u)\psi dx-\la\I\Om\I\Om\frac{(u+\hat{u})^{2_\mu^*}(u+\hat{u})^{2_\mu^*-1}\psi}{|x-y|^\mu}dxdy=0.
	\end{equation}
\end{Lemma}
\begin{proof}
	Let $0 \leq \psi \in X_0$. Then by Lemma \ref{l2.6}, there exists a sequence $\{\psi_k\}_{k \in \mb N}\subset X_0$ such that $\psi_k$ is increasing, each $\psi_k$ has compact support in $\Om$ and $\psi_k \ra \psi$ strongly in $X_0$. For each fixed $k$, we can find a sequence $\{\varphi_l^k\}_{l \in \mb N} \subset C_c^\infty(\Om)$ such that $\varphi_l^k \geq0$, $\bigcup_l \text{supp} \varphi_l^k$ is contained in a compact subset of $\Om$, $\{|\varphi_l^k|_\infty\}$ is bounded and $\varphi_l^k \ra \psi_k$ strongly as $l \ra \infty$. Since $u$ is a weak solution of $(\hat{P}_\la)$, we get 
		\begin{equation*}
		\ld u,\varphi_l^k\rd+\I\Om f(x,u)\varphi_l^k dx-\la\I\Om\I\Om\frac{(u+\hat{u})^{2_\mu^*}(u+\hat{u})^{2_\mu^*-1}\varphi_l^k}{|x-y|^\mu}dxdy=0.
	\end{equation*}
Now by the strong convergence of $\varphi_l^k \ra \psi_k$ in $X_0$ as $l \ra\infty$, we deduce that
\begin{equation*}
	\ld u,\psi_k\rd+\I\Om f(x,u)\psi_k dx-\la\I\Om\I\Om\frac{(u+\hat{u})^{2_\mu^*}(u+\hat{u})^{2_\mu^*-1}\psi_k}{|x-y|^\mu}dxdy=0.
\end{equation*}
Now using the monotone convergence theorem, dominated Convergence theorem and the strong convergence of $\psi_k$ in $X_0$, we obtain $f(x,u)\psi \in L^1(\Om)$ and we have \eqref{e2.4} for any $0\leq \psi \in X_0$. Finally, the result for general $\psi \in X_0$ holds due to the fact that $\psi =\psi^+-\psi^-$ and both $\psi^+$ and $\psi^-$ are nonnegative members of $X_0$. \QED
	\end{proof}
\subsection{Notion of Nonsmooth analysis}
To obtain the existence of nontrivial solutions to problem $(P_\la)$, we use some nonsmooth analysis tools. In this subsection we collect some basic definitions, observations and recall  a version of the linking theorem adapted to nonsmooth functionals. We begin with the following definition:
\begin{Definition}
	Let $V$ be a Hilbert space and $I: V \ra (-\infty,\infty]$ be a proper (i.e. $I \not \equiv \infty)$ lower semicontinuous functional.
	\begin{enumerate}[label=(\roman*)]
		\item Let $D(I)=\{u \in V:I(u)<\infty\}$ be the domain of $I$. For every $u \in D(I)$, we define the Fr\'{e}chet subdifferential of $I$ at $u $ as the set
		\begin{equation*}
			\partial^-I(u)=\left\{z \in V:\liminf_{v \ra u}\frac{I(v)-I(u)-\ld z,v-u\rd_V}{\|v-u\|_V}\geq 0\right\}.
		\end{equation*}
	\item For each $u \in V$, we define
	\begin{equation*}
		||| \partial^-I(u)|||=\begin{cases}
 			\min\{\|z\|_V:z \in \partial^-I(u)\}~&\text{if}~ \partial^-I(u)\ne\emptyset,\\
			\infty &\text{if}~\partial^-I(u)=\emptyset.
		\end{cases}
	\end{equation*}
	\end{enumerate}
	\end{Definition}
We know that $\partial^-I(u)$ is a closed convex set which may be empty. If $u \in D(I)$ is a local minimizer for $I$, then it can be seen that $0 \in \partial^-I(u)$.
\begin{Remark}
	We remark that if $I_0:V\ra (-\infty,\infty]$ is a proper, lower semicontinuous, convex functional, $I_1:V\ra \mb R$ is a $C^1$ functional, and $I=I_1 +I_0$, then $\partial^-I(u)=\nabla I_1(u)+\partial I_0 (u)$ for every $u \in D(I)=D(I_0)$, where $\partial I_0$ denotes the usual subdifferential of the convex functional $I_0$. Thus, $u$ is said to be a critical point of $I$ if $u \in D(I_0)$ and for every $v \in V$, we have $\ld \nabla I_1(u),v-u\rd_V +I_0(v)-I_0(u)\geq 0$.
\end{Remark}

\begin{Definition}
	For a proper, lower semicontinuous functional $I: V \ra (-\infty,\infty]$, we say that $I$ satisfies Cerami's variant of the Palais-Smale condition at a level $d$ (in short, $I$ satifies $(CPS)_d$), if any sequence $\{w_k\}_{k\in \mb N}\subset D(I)$ such that $I(w_k) \ra d$ and $(1+w_k)||| \partial^-I(w_k)||| \ra 0$ has a strongly convergent subsequence in $V$. 
\end{Definition}
Analogous to the mountain pass theorem, we have the following linking theorem for nonsmooth functionals.
\begin{Theorem}\label{t2.13}\cite{S} Let $V$ be a Hilbert space. Assume $I=I_0+I_1$, where $I_0:V\ra (-\infty,\infty]$ is a proper, lower semicontinuous, convex functional and $I_1:V\ra \mb R$ is a $C^1-$ functional. Let $B^n,~S^{n-1}$ denote the closed unit ball and its boundary in $\mb R^n$, respectively. Let $\varphi:S^{n-1}\ra D(I)$ be a continuous function such that 
	\begin{equation*}
		\Sigma=\{\psi\in C(B^n,D(I)):\psi\arrowvert_{S^{n-1}}=\varphi\}\ne\emptyset.
	\end{equation*}
Let $A$ be a relatively closed subset of $D(I)$ such that
\begin{equation*}
	A\cap\varphi(S^{n-1})=\emptyset,~A\cap \psi(B^n)\ne \emptyset~\text{for all}~\psi \in \Sigma,~\text{and}~\inf I(A)\geq \sup I(\varphi(S^{n-1})).
\end{equation*}	
Define $d =\inf_{\psi \in \Sigma}\sup_{x\in B^n}I(\psi(x))$. Assume that $d$ is finite and that $I$ satisfies $(CPS)_d$. Then there exists $u \in D(I)$ such that $I(u)=d$ and $0\in \partial^-I(u)$. Furthermore, if $\inf J(A)=d$, then there exists $u \in A\cap D(I)$ such that $I(u)=d$ and $0 \in \partial^-I(u)$.
\end{Theorem}
\section{Regularity of weak solutions and comparison principle} \label{s3}
In this section, we prove regularity results about nonnegative weak solutions to $({P}_\la)$. For this, we first investigate the regularity of nonnegative weak solution to $(\hat{P}_\la)$. We start with the $L^\infty$ estimates obtained by Moser type iterations:
\begin{Lemma}\label{lib}
	Any nonnegative solution to $(\hat{P}_\la)$ belongs to $L^\infty(\mb R^n)$.
\end{Lemma}
\begin{proof}
	Let $u$ be a nonnegative solution to $(\hat{P}_\la)$. We define $u_\tau =\min\{u,\tau\}$ for $\tau> 0$. Let $\psi =u(u_\tau)^{q-2}$, $q \geq 3$ be a test function to problem $(\hat{P}_\la)$. Now
	\begin{align*}
		\na(u u_\tau^{\frac{q}{2}-1})=u_\tau^{\frac{q}{2}-1}\na u +\left(\frac{q}{2}-1\right)u_\tau^{{\frac{q}{2}-2}}u\na u_\tau.
	\end{align*}
This implies
\begin{align*}
	\left|\na (u(u_\tau)^{\frac{q}{2}-1})\right|^2=&\sum_{i=1}^{n}\left(u_\tau^{\frac{q}{2}-1}\frac{\partial u}{\partial x_i}+\left(\frac{q}{2}-1\right)u_\tau^{{\frac{q}{2}-2}}u\frac{\partial u_\tau}{\partial x_i}\right)^2\leq 2\left(u_\tau^{q-2}|\na u|^2+\frac{q^2}{4}u_\tau^{q-4}u^2|\na u_\tau|^2\right)\\
	\leq& \frac{q^2}{2}\left(u_\tau^{q-2}|\na u|^2+u_\tau^{q-4}u^2|\na u_\tau|^2\right).
\end{align*}
Thus,
\begin{equation}\label{e3.1}
	\I\Om 	\left|\na (u(u_\tau)^{\frac{q}{2}-1})\right|^2 \leq \frac{q^2}{2}\left(\I\Om u_\tau^{q-2}|\na u|^2+\I{\{u<\tau\}}u^{q-2}|\na u_\tau|^2\right).
\end{equation}
Also we have
\begin{align}\nonumber\label{e3.2}
\I\Om \na u\na \psi =	\I\Om \na u\cdot\na (u(u_\tau)^{q-2})=&\I\Om u_\tau^{q-2}|\na u|^2+(q-2)\I\Om u_\tau^{q-3}u \na u\cdot \na u_\tau\\
	\geq&\I\Om u_\tau^{q-2}|\na u|^2 +\I{\{u<\tau\}}u^{q-2}|\na u|^2.
\end{align}
Combining \eqref{e3.1} and \eqref{e3.2}, we get
\begin{equation}\label{e3.3} 
	\I\Om 	\left|\na (u(u_\tau)^{\frac{q}{2}-1})\right|^2 \leq Cq^2 \I\Om \na u\na\psi.
\end{equation}
Now from \cite[Lemma 3.5]{GGS}, we have the following inequality:
\begin{equation}\label{e3.4}
	\frac{4(q-1)}{q^2}\left(a|a_\tau|^{\frac{q}{2}-1}-b|b_\tau|^{\frac{q}{2}-1}\right)^2\leq(a-b)(a|a_\tau|^{q-2}-b|b_\tau|^{q-2}).
\end{equation}
where $a,~b \in \mb R$, $q \geq 2$, $a_\tau=\min\{a,\tau\}$ and $b_\tau=\min\{b,\tau\}$. Using \eqref{e3.4} with $a =u(x)$ and $b=u(y)$, we obtain
\begin{equation}\label{equation3.5}
	[u(u_\tau)^{\frac{q}{2}-1}]_s^2 \leq \frac{Cq^2}{q-1}\I{\mb R ^n}\I{\mb R^n}\frac{u(x)-u(y)(\psi(x)-\psi(y))}{|x-y|^{n+2s}}dxdy.
\end{equation}
Using \eqref{e3.3}, \eqref{equation3.5} and Sobolev inequality, we get
\begin{align*}
	|u(u_\tau)^{\frac{q}{2}-1}|_{2^*}^2 \leq& C\left(\ll u(u_\tau)^{\frac{q}{2}-1}\gg^2+	[u(u_\tau)^{\frac{q}{2}-1}]_s^2\right)\leq Cq^2\ld u,\psi\rd\\
	=&Cq^2\left(-\I\Om f(x,u)u(u_\tau)^{q-2}dx+\I\Om\I\Om\frac{(u+\hat{u})^{2_\mu^*}(u+\hat{u})^{2_\mu^*-1}u(u_\tau)^{q-2}}{|x-y|^\mu}dxdy\right).
\end{align*}
The rest of the proof follows similarly as the proof of \cite[Lemma 4.1]{GGS}.\QED
	\end{proof}
\begin{Lemma}
	Let $q>0$ and let $v \in L^{(q+1)/q}(\Om)$ be a positive function and $u \in X_0 \cap L^{q+1}(\Om)$ a positive weak solution to
	\begin{equation*}
	\mc Mu+f(x,u)=v~\text{in}~\Om,~u=0~\text{in}~\mb R^n\setminus \Om.
		\end{equation*}
	Then $(u+\hat{u}-\e)^+ \in X_0$ for every $\e>0$. 
\end{Lemma}
\begin{proof}
	Let $\e_1,~\e_2>0$ and set $\varphi =\min\{u,\e_1-(\hat{u}-\e_2)^+\} \in X_0$. Note that $u -\varphi =(u+(\hat{u}-\e_2)^+-\e_1)^+ \in X_0$. Since
	\begin{equation*}
		0\leq v(u-\varphi)\leq vu +v\hat{u} \in L^1(\Om),
	\end{equation*}
and using the arguments as  in the proof of Lemma \ref{l2.7}, we can show that $f(\cdot,u)(u-\varphi) \in L^1(\Om)$ and 
\begin{equation*}
	\ld u,u-\varphi\rd+\I\Om f(x,u)(u-\varphi) dx-\I\Om v(u-\varphi)=0.
\end{equation*}
Now using the following inequality for the fractional Laplacian:
\begin{equation*}
	(-\De)^s g(u) \leq g^{\prime}(u)(-\De)^s u,
\end{equation*}
where $g$ is a convex piecewise $C^1$ with bounded derivative  function, we have
\begin{align*}
	\ld (\hat{u}-\e_2)^+, \psi\rd \leq 	\ld \hat{u}, \psi\rd=\I\Om\hat{u}^{-\gamma}\psi dx,~\text{for every}~0\leq \psi \in C_c^\infty(\Om).
\end{align*}
So, arguing as in the proof of Lemma \ref{l2.7}, we can show that
\begin{equation*}
	\ld (\hat{u}-\e_2)^+,u-\varphi \rd \leq \I\Om\hat{u}^{-\gamma}(u-\varphi) dx.
\end{equation*}
We note that $u+\hat{u} \geq \e_1$ when $u \ne \varphi$, $(u+\hat{u})^{-\gamma}(u-\varphi) \in L^1(\Om)$ and $\hat{u}(u-\varphi) \in L^1(\Om)$. Therefore, we have
\begin{align*}
	\| (u+(\hat{u}-\e_2)^+-\e_1)^+\|^2=&\ld  (u+(\hat{u}-\e_2)^+-\e_1)^{+},u-\varphi\rd\\
	\leq&\I\Om \hat{u}^{-\gamma}(u-\varphi)dx-\I\Om f(x,u)(u-\varphi)dx+\I\Om v(u-\varphi)dx\\
	=& \I\Om(u+\hat{u})^{-\gamma}(u-\varphi)dx +\I\Om v(u-\varphi)dx\\
	\leq& \e_1^{-\gamma}\I\Om (u-\psi)dx +\I\Om v(u-\varphi)dx.
\end{align*}
Thus for any $\e_1>0$, we have that $(u+(\hat{u}-\e_2)^+-\e_1)^+$ is bounded in $X_0$ as $\e_2 \ra 0^+$. Hence, we conclude that $(u+\hat{u}-\e)^+ \in X_0$ for every $\e >0$. The second assertion follows from assertion (iii)  of theorem \ref{ART}.\QED
	\end{proof}
\begin{Corollary}\label{c3.3}
	 Let $v \in L^{2^*}(\Om)$ be a positive function and assume $g(x,v) =\left(\ds\I\Om\frac{v^{2_\mu^*}(y)}{|x-y|^\mu}dy\right)v^{2_\mu^*-1}$. Assume that $u \in X_0$ be a positive weak solution to
	 \begin{equation}\label{eq3.7}
	 	\mc M u +f(x,u)=g(x,v)~\text{in}~\Om,~u=0~\text{in}~\mb R^n\setminus \Om.
	 \end{equation}
 Then $(u+\hat{u}-\e)^+ \in X_0$ for every $\e >0$.
\end{Corollary}
Now we establish a crucial comparison principle. It states as follows:
\begin{Lemma}\label{lcp}
Let $\mc H \in X_0^*$ (the dual of $X_0$) and let $v,~w \in H^1_\text{loc}(\Om)$ be such that $v,~w>0$ a.e in $\Om$, $v,~w \geq 0$ in $\mb R^n$, $v^{-\gamma},~w^{-\gamma} \in L^1_{\text{loc}}(\Om)$, $(v-\e)^+ \in X_0$ for all $\e>0$, $z \in L^1(\Om)$ and
\begin{equation}\label{e3.7} 
	\ld v,\psi\rd \leq \I\Om v^{-\gamma}\psi dx +(\mc H,\psi),~\ld w,\psi\rd \geq \I\Om w^{-\gamma}\psi dx +(\mc H,\psi),
\end{equation}
for all compactly supported function $0\leq\psi \in X_0\cap L^\infty(\Om)$. Then $v \leq w$ a.e in $\Om$.
\end{Lemma}
\begin{proof}
	Let us denote $\Psi_k:\mb R \ra \mb R$ as the primitive of the function
	\begin{equation*}
		\tau \mapsto \begin{cases}
			\max\{-\tau^{-q},-k\}~&\text{if}~\tau>0,\\
			-k~&\text{if}~\tau\leq 0
		\end{cases}
	\end{equation*}
such that $\Psi_k(1)=0$. Next we define a proper lower semicontinuous, strictly convex functional $\widetilde{G}_{0,k}:L^2(\Om)\ra \mb R$ as
\begin{equation*}
	\widetilde{G}_{0,k}(u) =\begin{cases}
		\frac12 \|u\|^2+\I\Om \Psi_k(u)dx~&\text{if}~u \in X_0,\\
		\infty~&\text{if}~u \in L^2(\Om)\setminus X_0.
	\end{cases}
\end{equation*}
As we know, primitives are usually defined up to an additive constant, to prevent a possible unlikely choice we consider $G_{0,k}:L^2(\Om) \ra \mb R$ defined by
\begin{equation*}
	G_{0,k}(u)=\widetilde{G}_{0,k}(u)-\min \widetilde{G}_{0,k}=\widetilde{G}_{0,k}(u)-\widetilde{G}_{0,k}(u_{0,k}),
\end{equation*}
where $u_{0,k} \in X_0$ is the minimum of $\widetilde{G}_{0,k}$. In general, for $\mc H \in X_0^*$ we set
\begin{equation*}
	\widetilde{G}_{\mc H,k}(u)=\begin{cases}
		G_{0,k}(u)-(\mc H, u-u_{0,k})~&\text{if}~u \in X_0\\
		\infty~&\text{if}~u \in L^2(\Om)\setminus X_0.
		\end{cases}
\end{equation*}
Let $\e>0$, $k >\e^{-\gamma}$, and let $z$ be the minimum of the functional $\widetilde{G}_{\mc H,k}$ on the convex set $K =\{\psi\in X_0:0\leq \psi \leq w~\text{a.e in}~\Om\}$. Then for all $\psi \in K$ we get
\begin{equation}\label{e3.8}
	\ld z,\psi-z\rd \geq -\I\Om \Psi_k^{\prime}(z)(\psi-z)dx+(\mc H,\psi-z). 
\end{equation}
Let $0\leq \psi \in C_c^\infty(\Om)$, $t>0$. Define $\psi_t:=\min\{z+t\psi,w\}$. Noting that $w \in H^1_{\text{loc}}(\Om)$, $z \in X_0$, $\psi \in C_c^\infty(\Om)$, we have $\psi_t \in X_0$. Next we claim that $\psi_t$ is uniformly bounded in $X_0$ for all $t<1$. Using the continuous embedding of $H_0^1(\Om)$ into $H^s(\mb R^n)$, it is sufficient to show that $\ll\psi_t\gg$ is uniformly bounded in $t$. We have 
\begin{align*}
	\I\Om |\na \psi_t|^2 =&\I{\{z+t\psi\leq w\}} |\na z+t\na\psi|^2+\I{\{z+t\psi> w\}}|\na w|^2\\
	\leq& \I\Om |\na z|^2 +t^2 \I\Om |\na \psi|^2+2t \I\Om \na z \na \psi +\I{\text{supp}\psi}|\na w|^2\\
	\leq& \ll z \gg^2 +\ll \psi \gg^2 +\ll z\gg \ll \psi \gg +\I{\text{supp}\psi}|\na w|^2 <\infty.
\end{align*}
This proves the claim. Considering the subsequence (still denoted by $\psi_t$) such that $\psi_t \rightharpoonup z$ weakly in $X_0$ and taking $\psi=\psi_t$ in \eqref{e3.8}, we obtain
\begin{equation}\label{e3.9}
	\ld z, \psi_t -z \rd \geq  -\I\Om \Psi_k^{\prime}(z)(\psi_t-z)dx+(\mc H,\psi_t-z).
\end{equation}
Since $w$ is a supersolution and $w^{-\gamma}\geq -\Psi_k^{\prime}(w)$, we infer that $w$ satisfies
\begin{equation}\label{e3.10}
	\ld w, \psi\rd \geq -\I\Om \Psi_k^{\prime}(w)\psi dx +(\mc H, \psi).
\end{equation}
Using the facts that $\psi_t \leq w$, $\psi_t -z-t\psi \leq 0$ and $\psi_t -z-t\psi \ne 0$ only if $\psi_t =w$, we observe that
\begin{align}\nonumber\label{e3.11}
	&\I\Om \na \psi_t \na (\psi_t -z-t\psi) +\frac{C(n,s)}{2}\I{\mb R^n}\I{\mb R^n} \frac{(\psi_t(x)-\psi_t(y))((\psi_t -z-t\psi)(x)-(\psi_t -z-t\psi)(y))}{|x-y|^{n+2s}}dxdy\\ \nonumber
	\leq& \I\Om \na w \na (\psi_t -z-t\psi) +\frac{C(n,s)}{2}\I{\mb R^n}\I{\mb R^n}\frac{w(x)(\psi_t -z-t\psi)(x)}{|x-y|^{n+2s}}dxdy\\ \nonumber
	&+\frac{C(n,s)}{2}\I{\mb R^n}\I{\mb R^n}\frac{w(y)(\psi_t -z-t\psi)(y)}{|x-y|^{n+2s}}dxdy
	-\frac{C(n,s)}{2}\I{\mb R^n}\I{\mb R^n}\frac{w(x)(\psi_t -z-t\psi)(y)}{|x-y|^{n+2s}}dxdy\\
	&-\frac{C(n,s)}{2}\I{\mb R^n}\I{\mb R^n}\frac{w(y)(\psi_t -z-t\psi)(x)}{|x-y|^{n+2s}}dxdy
	= \ld w, \psi_t -z-t\psi\rd.
\end{align}
Similarly, $\ds \I\Om (\Psi_k^{\prime}(\psi_t)-\Psi_k^{\prime}(w))(\psi_t -z-t\psi) \leq 0$ and moreover $\Psi_k^{\prime}(w) \leq -w^{-\gamma}$. Taking into account \eqref{e3.7}, \eqref{e3.9}, \eqref{e3.10}, \eqref{e3.11} and above observations, we infer that
\begin{align*}
	\| \psi_t -z\|^2 &-\I\Om(-\Psi_k^{\prime}(\psi_t)+\Psi_k^{\prime}(z))(\psi_t-z)dx\\
	 =& \ld \psi_t,\psi_t-z\rd +\I\Om \Psi_k^{\prime}(\psi_t)(\psi_t-z)dx -\ld z,\psi_t-z\rd
	-\I\Om\Psi_k^{\prime}(z)(\psi_t -z)dx \\
	\leq&   \ld \psi_t,\psi_t-z\rd +\I\Om \Psi_k^{\prime}(\psi_t)(\psi_t-z)dx -(\mc H, \psi_t-z)\\
\end{align*}
\begin{align*}
	=&\ld \psi_t,\psi_t-z-t\psi\rd +\I\Om \Psi_k^{\prime}(\psi_t)(\psi_t-z-t\psi)dx -(\mc H, \psi_t-z-t\psi)\\
	&+t\left(\ld \psi_t,\psi\rd+\I\Om \Psi_k^{\prime}(\psi_t)\psi-(\mc H,\psi)\right)\\
	\leq&\ld w,\psi_t-z-t\psi\rd +\I\Om \Psi_k^{\prime}(w)(\psi_t-z-t\psi)dx -(\mc H, \psi_t-z-t\psi)\\
	&+t\left(\ld \psi_t,\psi\rd+\I\Om \Psi_k^{\prime}(\psi_t)\psi-(\mc H,\psi)\right)
	\leq t\left(\ld \psi_t,\psi\rd+\I\Om \Psi_k^{\prime}(\psi_t)\psi-(\mc H,\psi)\right).
\end{align*}
This yields
\begin{align*}
	\ld \psi_t,\psi\rd+\I\Om \Psi_k^{\prime}(\psi_t)\psi-(\mc H,\psi)\geq& \frac{1}{t}\left(\| \psi_t -z\|^2 -\I\Om|\Psi_k^{\prime}(\psi_t)-\Psi_k^{\prime}(z)|(\psi_t-z)dx\right)\\
	\geq&-\I\Om |\Psi_k^{\prime}(\psi_t)-\Psi_k^{\prime}(z)|\psi dx.
\end{align*}
Now using the weak convergence of $\psi_t$, monotone convergence theorem and dominated convergence theorem, we have
\begin{equation}\label{e3.12}
	\ld z,\psi\rd \geq -\I\Om \Psi_k^{\prime}(z)\psi dx +(\mc H,\psi). 
\end{equation}
Since $C_c^\infty(\Om)$ is dense in $X_0$, we infer that \eqref{e3.12} is true for all nonnegative $\psi \in X_0$. In particular, since $z \geq 0$ we have $(v-z-\e)^+ \in X_0$. Testing \eqref{e3.12} with  $(v-z-\e)^+$, we get 
\begin{equation}\label{e3.13}
	\ld z,(v-z-\e)^+\rd \geq -\I\Om \Psi_k^{\prime}(z)(v-z-\e)^+ dx +(\mc H,(v-z-\e)^+).
\end{equation}
Let us now consider $\Theta \in X_0$ such that $0 \leq \Theta \leq v$ a.e. in $\Om$. Let $\{\Theta_m\}$ be a sequence in $C_c^\infty(\Om)$ converging to $\Theta \in X_0$ and set $\tilde{\Theta}_m =\min\{{\Theta}_m^+,\Theta\}$. Testing \eqref{e3.7} with $\tilde{\Theta}_m$, we get
\begin{equation*}
	\ld v,\tilde{\Theta}_m\rd \leq \I\Om v^{-\gamma}\tilde{\Theta}_mdx +(\mc H, \tilde{\Theta}_m).
\end{equation*}
If $v^{-\gamma}\Theta \in L^1(\Om)$, then passing to the limit as $m \ra \infty$, we get
\begin{equation*}
	\ld v,{\Theta}\rd \leq \I\Om v^{-\gamma}{\Theta}dx +(\mc H, {\Theta}).
\end{equation*}
If $v^{-\gamma}\Theta \not\in L^1(\Om)$, then the above inequality is obviously still true. In particular, we have
\begin{equation}\label{e3.14}
	\ld v,(v-z-\e)^+\rd \leq \I\Om v^{-\gamma}(v-z-\e)^+dx +(\mc H, (v-z-\e)^+).
\end{equation}
Using \eqref{e3.13}, \eqref{e3.14} together with the fact that $k \geq \e^{-\gamma}$, we get
\begin{align*}
	\ld (v-z-\e)^+,(v-z-\e)^+\rd \leq& \ld v-z,(v-z-\e)^+\rd
	\leq  \I\Om (v^{-\gamma}+\Psi_k^{\prime}(z))(v-z-\e)^+dx\\
	=&  \I\Om (-\Psi_k^{\prime}(v)+\Psi_k^{\prime}(z))(v-z-\e)^+dx\leq 0.
\end{align*} 
Thus $v\leq z+\e \leq w+\e.$ Since $\e$ was arbitrary chosen, the proof follows.\QED
	\end{proof}
\begin{Lemma}\label{l3.5}
	Let $\la >0$ and let $ v $ be a weak solution to $(P_\la)$ as it is defined in Definition \ref{d2.3}. Then $v-\hat{u}$ is a positive weak solution to $(\hat{P}_\la)$ belonging to $L^\infty(\Om)$.
\end{Lemma}
\begin{proof}
	Consider problem \eqref{eq3.7} with $v$ given. Then $0$ is a strict subsolution to \eqref{eq3.7}. Define the functional $J : X_0 \ra(-\infty,\infty]$ by
	\begin{equation*}
		J(u)= \begin{cases}
			\ds \frac12\|u\|^2+\I\Om F(x,u)dx -\frac{\la}{2 2_\mu^\ast}\I\Om\I\Om \frac{v^{2_\mu^\ast} v^{2_\mu^\ast-1}u}{|x-y|^\mu}dxdy~&\text{if}~F(\cdot,u)\in L^1(\Om),\\
			\infty~&\text{otherwise}.
		\end{cases}
	\end{equation*}
Define $K^{\prime} =\{u \in X_0: u\geq 0\}$, a closed convex set  and define
 \begin{equation*}
 	J_{K^{\prime}}(u)= \begin{cases}
 		J(u)~&\text{if}~u \in K^{\prime}~\text{and}~F(\cdot,u)\in L^1(\Om),\\
 		\infty~&\text{otherwise}.
 	\end{cases}
 \end{equation*}
We can easily show that there exists $u \in K^{\prime}$ such that $J_{K^{\prime}}(u) =\inf _{w\in K^\prime}J_{ K^{\prime}}(w)$. This implies that $0 \in \partial^- J_{K^{\prime}}(u)$. From Proposition \ref{p4.2}, we obtain that $u$ is a nonnegative solution to \eqref{eq3.7}. Using Corollary \ref{c3.3}, Lemma \ref{l2.5} and Lemma \ref{l2.7}, we obtain that $(u+\hat{u}-\e)^+ \in X_0$ for every $\e >0$ and 
\begin{equation*}
	\ld u+\hat{u}, \psi \rd -\I\Om (u+\hat{u})^{-\gamma}\psi dx -\la\I\Om\I\Om \frac{v^{2_\mu^\ast} v^{2_\mu^\ast-1}\psi}{|x-y|^\mu}dxdy=0,
\end{equation*}
\begin{equation*}
	\ld v, \psi \rd -\I\Om v^{-\gamma}\psi dx -\la\I\Om\I\Om \frac{v^{2_\mu^\ast} v^{2_\mu^\ast-1}\psi}{|x-y|^\mu}dxdy=0
\end{equation*}
for $0\leq\psi \in X_0 \cap L^\infty(\Om)$ with compact support in $\Om$. Now using Lemma \ref{lcp}, we get $v =u+\hat{u}$, which implies that $u =v-\hat{u}$ is a positive weak solution of $(\hat{P}_\la)$. Finally, by Lemma \ref{lib}, we have $u \in L^\infty({\mb R^n})$. \QED
	\end{proof}
\begin{Lemma}\label{l3.6}
	Let $\mu < \min\{4,n\}$. Let $u$ be any weak solution to $(P_\la)$. Then $u \in L^\infty(\Om)\cap C^{0,\alpha}_{loc}(\Om)\cap \mc G(\Om)$, for some $\al \in (0,1)$. If $\gamma<3$ then $u\in X_0\cap C^{1,\beta}_{loc}(\Om)$ for some $\beta\in (0,1)$.
\end{Lemma}
\begin{proof}
	Let $u$ be any weak solution of problem $(P_\la)$. Using Lemma \ref{l3.5}, we have $u-\hat{u} \in X_0$ is a solution of $(\hat{P}_\la)$. Again using Lemma \ref{lib}, we have $ u-\hat{u} \in L^\infty(\mb R^n)$. Therefore, $u=(u-\hat{u})+\hat{u} \in L^\infty(\mb R^n)$. Now let $\tilde{u}$ be a solution to the following problem:
	\begin{align*}
		\mc M \tilde{u} =\tilde{u}^{-\gamma} +\la d, ~\tilde{u} >0 \in \Om,~ \tilde{u} =0~\text{in}~\mb R^n\setminus \Om
	\end{align*}
where $d =D^* |u|_\infty^{22_\mu^*-1}$ with $D^* =\left|\ds\I\Om \frac{dy}{|x-y|^\mu} \right|_\infty$. Using Lemma \ref{lcp}, we observe that $ \hat{u} \leq u\leq \tilde{u}$ a.e in $\Om$. Finally by using the regularity of $\hat{u}$ and $\tilde{u}$, we conclude that $u \in \mc G(\Om)$. Next we show $C^{0,\al}_{\text{loc}}$ regularity of $u$, for some $\al \in (0,1)$. For this, noting that since $0<\mu<n$ and $\Om$ is bounded, we have
\begin{align*}
	\left|\I\Om\frac{|u(y)|^{2_\mu^\ast}}{|x-y|^\mu}dy\right|\leq& |u|_\infty^{2_\mu^\ast}\left[\I{\Om \cap \{|x-y|<1\}}\frac{dy}{|x-y|^\mu}+\I{\Om \cap \{|x-y|\geq1\}}\frac{dy}{|x-y|^\mu}\right]\\
	\leq& |u|_\infty^{2_\mu^\ast}\left[\I{\Om \cap \{r<1\}}r^{n-1-\mu}+|\Om|\right] <\infty.
\end{align*} 
Hence the right hand side of $(P_\la)$ is in $L_{\text{loc}}^\infty(\Om)$. Then by \cite[Theorem 1.4]{GL}, we see that $u \in C^{0,\al}_{\text{loc}}(\Om)$ for some $\al \in (0,1)$. If $\gamma<3$, then $\hat{u}\in X_0$ and then $u\in X_0$. Furthermore, using Theorem 1 in \cite{FM}, local H\"older regularity of $\nabla u$ follows.\QED
\end{proof}
We complete this section by giving \\
\textbf{Proof of Remark \ref{remark1.2}:} To prove part (1), we shall make use of \cite[Theorem 1]{FM} i.e. we will show that the right hand side of $(P_\la)$ is in $L^n(\Om)$. Since $u \in L^\infty(\Om)$, we only need to show $u^{-\gamma} \in L^n(\Om)$. Indeed by using the boundary behaviour of $u$ and the restriction $0<\gamma<\frac{1}{n}$, we readily see that $u^{-\gamma} \in L^n(\Om)$. Hence by \cite[Theorem 1]{FM}, we obtain $u \in C^{0,\al}(\mb R^n)$ for all $\al \in (0,1)$.  Now using \cite[Proposition 2.5]{Si}, we see that $(-\De)^s u \in C^{0,\al-2s}(\Om)$ for $2s <\al<1$. Finally by using the elliptic regularity theory, we get $u \in C^{2,\al-2s}_{\text{loc}}(\Om)$. \\
 For part (2), by taking into consideration the boundary behaviour of $u$, we see that $u^{-\gamma} \in L^p(\Om)$ iff $p \in (1,\frac{1}{\gamma})$ and so the right hand side of $(P_\la)$ is in $L^p(\Om)$ iff $p \in (1,\frac{1}{\gamma})$. Finally, we conclude that $u \in W^{2,p}(\Om)$ in view of \cite[Theorem 1.4]{SVWZ}. \QED
\section{Existence, nonexistence and multiplicity of weak solutions } \label{s4}
	\subsection{First Solution}
	In this subsection, we prove the existence of a weak solution which actually comes out to be a local minimizer of an appropriate functional. We start this subsection by giving the variational framework to problem $(\hat{P}_\la)$ in the space $X_0$.  We define the functional $\Phi:X_0 \ra (-\infty,\infty]$ associated with $(\hat{P}_\la)$ by
\begin{equation*}
	\Phi(u)= \begin{cases}
		\ds \frac12\|u\|^2+\I\Om F(x,u)dx -\frac{\la}{2 2_\mu^\ast}\I\Om\I\Om \frac{|u+\hat{u}|^{2_\mu^\ast} |u+\hat{u}|^{2_\mu^\ast}}{|x-y|^\mu}dxdy~&\text{if}~F(\cdot,u)\in L^1(\Om),\\
		\infty~&\text{otherwise}.
	\end{cases}
\end{equation*}	
	Next for any closed convex subset $K \subset X_0$, we define the functional $\Phi_K :X_0 \ra (-\infty,\infty]$ by
	 \begin{equation*}
		\Phi_{K}(u)= \begin{cases}
			\Phi(u)~&\text{if}~u \in K~\text{and}~F(\cdot,u)\in L^1(\Om),\\
			\infty~&\text{otherwise}.
		\end{cases}
	\end{equation*}
We note that $u \in D(\Phi_K)$ iff $u \in K$ and $F(\cdot,u) \in L^1(\Om)$. Our next lemma characterizes the set $\partial^-\Phi_K(u)$. 
\begin{Lemma}\label{l4.1} 
	Let $K\subset X_0$ be a convex set and let $\vartheta \in X_0$. Let also $u \in K$ with $F(\cdot,u) \in L^1(\Om)$. Then the following assertions are equivalent:
	\begin{enumerate}[label=(\roman*)]
		\item $\vartheta \in \partial^-\Phi_K(u)$.
		\item For every $\varphi \in K$ with $F(\cdot,\varphi) \in L^1(\Om)$, we have $f(\cdot,u)(\varphi-u) \in L^1(\Om)$ and
		\begin{equation*}
		\ld \vartheta,\varphi -u\rd \leq 	\ld u,\varphi-u\rd +\I\Om f(x,u) (\varphi-u)dx - \la \I\Om\I\Om\frac{(u+\hat{u})^{2_\mu^\ast}(u+\hat{u})^{2_\mu^\ast-1}(\varphi-u)}{|x-y|^\mu}dxdy.
		\end{equation*}
	\end{enumerate}
\end{Lemma}
\begin{proof}
	The proof is similar to the proof of \cite[Lemma 5.1]{GGS1} and hence omitted.  \QED
	\end{proof}
Now for any functions $v,~w : \Om\ra [-\infty,\infty]$, we define the following convex sets:
\begin{align*}
	K_v=\{u \in X_0 :v \leq u~\text{a.e}\},~	K^w=\{u \in X_0 : u\leq w~\text{a.e}\}~\text{and}~	K_v^w=\{u \in X_0 :v \leq u\leq w~\text{a.e}\}.
\end{align*}	
We state the following proposition which can be thought of as Perron's method for non-smooth functionals.
\begin{Proposition}\label{p4.2}
	Assume that one of the following conditions holds:
	\begin{enumerate}[label=(\roman*)]
		\item $v_1$ is a subsolution to $(\hat{P}_\la)$, $F(x,\varphi(x)) \in L^1_{\text{loc}}(\Om)$ for all $\varphi \in K_{v_1}$, $u \in D(\Phi_{K_{v_1}})$  and \linebreak $0 \in \partial^-\Phi_{K_{v_1}}(u)$. 
		\item $v_2$ is a supersolution of $(\hat{P}_\la)$, $F(x,\varphi(x)) \in L^1_{\text{loc}}(\Om)$ for all $\varphi \in K^{v_2}$, $u \in D(\Phi_{K^{v_2}})$ and \linebreak $0\in \partial^-\Phi_{K^{v_2}}(u)$.
		\item $v_1$ and $v_2$ are subsolution and supersolution of $(\hat{P}_\la)$, $v_1 \leq v_2$, $F(x,v_1(x)),~F(x,v_2(x)) \in L^1_{\text{loc}}(\Om)$, $u \in D(\Phi_{K_{v_1}^{v_2}})$ and $0 \in \partial^-\Phi_{K_{v_1}^{v_2}}(u)$.
	\end{enumerate}
Then $u$ is a weak solution of $(\hat{P}_\la)$.
\end{Proposition}
\begin{proof}
	Following the proof of \cite[Proposition 4.2]{GMS}, we have the required result. \QED
	\end{proof}	
	Let $\xi$ be the function which satisfies $\mc M u =\frac12$. From \cite[Theorem 2.7]{BDVV4}, $\xi \in C^{1,\ba}(\bar{\Om})$ for some $\ba \in (0,1)$. For $f$ and $F$, we have the following properties.
	\begin{Lemma}\label{l4.3}
		\begin{enumerate}[label=(\roman*)]
			\item Let $u \in L^1_{\text{loc}}(\Om)$ such that ess $\inf_K u>0$ for any compact set $K \subset \Om$. Then $f(x,u(x)),~F(x,u(x)) \in L^1_{\text{loc}}(\Om)$.
			\item For all $x \in \Om$, the following holds:
			\begin{itemize}
				\item [(a)] $F(x,st) \leq s^2 F(x,t)$ for each $s \geq 1$ and $t \geq 0$.
				\item [(b)] $F(x,s) -F(x,t)-(f(x,s)+f(x,t))(s-t)/2 \geq 0$ for each $s,~t$ with $s \geq t >-\xi(x)$.
				\item [(c)] $F(x,s)-f(x,s)s/2 \geq 0$ for each $s \geq 0$.
			\end{itemize}
		\end{enumerate}
	\end{Lemma}
	\begin{proof}
	For a proof we refer to \cite[Lemma 4]{HSS}. 	
		\end{proof}
\begin{Lemma}\label{l4.4}
	The following hold:
	\begin{enumerate}[label=(\roman*)]
		\item $0$ is the strict subsolution to $(\hat{P}_\la)$ for all $\la >0$.
		\item $\xi$ is a strict supersolution to $(\hat{P}_\la)$ for all sufficiently small $\la>0$.
		\item Any positive weak solution $z$ to $(\hat{P}_{\la_2})$ is a strict supersolution to $(\hat{P}_{\la_1})$ for $0<\la_1<\la_2$.
	\end{enumerate}
\end{Lemma}
\begin{proof}
	\begin{enumerate}[label=(\roman*)]
		\item Let $\psi \in X_0\setminus\{0\}$, $\psi\geq 0$. Since $f(x,0)=0$, we get
		\begin{equation*}
			\ld 0,\psi\rd +\I\Om f(x,0)\psi -\la\I\Om\I\Om\frac{(0+\hat{u})^{2_\mu^*}(0+\hat{u})^{2_\mu^*-1}\psi}{|x-y|^\mu}dxdy<0.
		\end{equation*}
	\item Choose $\la$ small enough such that $\ds \la \left(\I\Om \frac{(\xi +\hat{u})^{2_\mu^*}}{|x-y|^\mu}dy\right)(\xi+\hat{u})^{2_\mu^*-1} <1$ in $\Om$. From Lemma \ref{l4.3}, $f(x,\xi)$, $F(x,\xi) \in L^1_{\text{loc}}(\Om)$, for all $\psi \in X_0 \setminus\{0\}$, we deduce that
	\begin{align*}
		\ld \xi,\psi\rd +\I\Om f(x,\xi)\psi dx-\la \I\Om\I\Om \frac{(\xi+\hat{u})^{2_\mu^\ast}(\xi+\hat{u})^{2_\mu^\ast-1}\psi}{|x-y|^\mu}dxdy\\
		\geq \I\Om \left(1- \la\left(\I\Om \frac{(\xi +\hat{u})^{2_\mu^*}}{|x-y|^\mu}dy\right)(\xi+\hat{u})^{2_\mu^*-1}\right)\psi dx >0.
	\end{align*}
\item Let $0 <\la_1 <\la_2$ and $z$ be a weak positive weak solution to $(\hat{P}_{\la_2})$. Then for all $\psi \in X_0\setminus\{0\}$, we have 
\begin{align*}
	\ld z,\psi\rd +\I\Om f(x,z)\psi -\la_1 \I\Om\I\Om \frac{(z+\hat{u})^{2_\mu^\ast}(z+\hat{u})^{2_\mu^\ast-1}\psi}{|x-y|^\mu}dxdy
\end{align*}
\begin{align*}
	=(\la_2-\la_1)\I\Om\I\Om \frac{(z+\hat{u})^{2_\mu^\ast}(z+\hat{u})^{2_\mu^\ast-1}\psi}{|x-y|^\mu}dxdy >0.
\end{align*}
	\end{enumerate}
This completes the proof. \QED
	\end{proof}
Let $\La:= \sup\{\la>0:(\hat{P}_\la)~\text{admits a solution}\}$.
\begin{Remark}
	If $\La >0$, by Lemma \ref{l4.4}, we deduce that for any $\la \in (0,\La)$, $(\hat{P}_\la)$ has a subsolution (the trivial function $0$) and a positive strict supersolution (say $z$).
\end{Remark}
\begin{Theorem}\label{t4.6}
	Let $v_1,~v_2 :\mb R^n \ra [-\infty,+\infty]$ with $v_1 \leq v_2$ such that $v_2$ is a strict supersolution to $(\hat{P}_\la)$ and $u \in D(\Phi_{K_{v_1}^{v_2}})$ be a minimizer for $\Phi_{K_{v_1}^{v_2}}$. Then $u$ is a local minimizer of $\Phi_{K_{v_1}}$.
	\end{Theorem}	
\begin{proof}
	For each $w \in K_{v_1}$ and $0 \leq \varphi \in X_0$, we define $\eta(w)=\min\{w,v_2\}= w-(w-v_2)^+$ and 
	\begin{equation*}
		\mc J (\varphi)= \ld v_2,\varphi\rd +\I\Om f(x,v_2)\varphi dx -\la\I\Om\I\Om \frac{(v_2+\hat{u})^{2_\mu^\ast}(v_2+\hat{u})^{2_\mu^\ast-1}\varphi}{|x-y|^\mu}dxdy.
	\end{equation*}
We first claim that
\begin{align*}
	\ld \eta (w), w-\eta (w)\rd \geq \ld v_2,w-\eta(w)\rd~\text{and}
\end{align*}
\begin{align*}
	\I\Om\I\Om\frac{\left((\eta(w)+\hat{u})^{2_\mu^*}(\eta(w)+\hat{u})^{2_\mu^*-1}-(v_2+\hat{u})^{2_\mu^*}(v_2+\hat{u})^{2_\mu^*-1}\right)(w-\eta(w))}{|x-y|^\mu}dxdy \leq 0.
\end{align*}
Let $\Om^{\prime}=\text{supp}((w-v_2)^+)$. Then on $\Om^{\prime}$, $\eta(w)=v_2$ and using the fact that $\eta(w)\leq v_2$ on $\Om$, we easily deduce that	
	\begin{equation*}
		\ld \eta(w),w-\eta(w)\rd \geq \ld v_2,w-\eta(w)\rd.
	\end{equation*}
Also the second inequality hold using the fact that $\eta(w)\leq v_2$ on $\Om$. This proves the claim. Deploying the fact that $u$ is a minimizer for $\Phi_{K_{v_1}^{v_2}}$, $\eta(w) \in D(\Phi_{K_{v_1}^{v_2}})$, \cite[Lemma 2]{HSS} and using that $F(x,\cdot)$ is convex, we have
	\begin{align}\nonumber\label{e4.2}
	\Phi_{K_{v_1}}(w)&-\Phi_{K_{v_1}}(u) \geq \Phi_{K_{v_1}}(w)-\Phi_{K_{v_1}}(\eta(w))\\ \nonumber
	=&\frac{\|w-\eta(w)\|^2}{2}+\ld \eta(w),w-\eta(w)\rd+\I\Om (F(x,w)-F(x,\eta(w)))\\ \nonumber
	&-\frac{\la}{22_\mu^*}	\I\Om\I\Om\frac{\left((w+\hat{u})^{2_\mu^*}(w+\hat{u})^{2_\mu^*}-(\eta(w)+\hat{u})^{2_\mu^*}(\eta(w)+\hat{u})^{2_\mu^*}\right)}{|x-y|^\mu}dxdy\\ \nonumber
		\geq&\frac{\|w-\eta(w)\|^2}{2}+\ld \eta(w),w-\eta(w)\rd+\I\Om f(x,\eta(w))(w-\eta(w))\\ \nonumber
	&-\frac{\la}{22_\mu^*}	\I\Om\I\Om\frac{\left((w+\hat{u})^{2_\mu^*}(w+\hat{u})^{2_\mu^*}-(\eta(w)+\hat{u})^{2_\mu^*}(\eta(w)+\hat{u})^{2_\mu^*}\right)}{|x-y|^\mu}dxdy\\ \nonumber
		\geq&\frac{\|w-\eta(w)\|^2}{2}+\ld v_2,w-\eta(w)\rd+\I\Om f(x,v_2)(w-\eta(w))\\ \nonumber
	&-\frac{\la}{22_\mu^*}	\I\Om\I\Om\frac{\left((w+\hat{u})^{2_\mu^*}(w+\hat{u})^{2_\mu^*}-(\eta(w)+\hat{u})^{2_\mu^*}(\eta(w)+\hat{u})^{2_\mu^*}\right)}{|x-y|^\mu}dxdy\\
	\geq &\frac{\|w-\eta(w)\|^2}{2} +\mc J(w -\eta(w))-\frac{\la}{22_\mu^*} \mc D,
\end{align}
where
\begin{align*}
	\mc D =&\I\Om\I\Om \frac{(w+\hat{u})^{2_\mu^*}(w+\hat{u})^{2_\mu^*}}{|x-y|^\mu}dxdy -\I\Om\I\Om \frac{(\eta(w)+\hat{u})^{2_\mu^*}(\eta(w)+\hat{u})^{2_\mu^*}}{|x-y|^\mu}dxdy \\
	&-22_\mu^*\I\Om\I\Om \frac{(\eta(w)+\hat{u})^{2_\mu^*}(\eta(w)+\hat{u})^{2_\mu^*-1}(w-\eta(w))}{|x-y|^\mu}dxdy.
\end{align*}
	Next we estimate $\mc G$ from above. For this, we first note that
	\begin{align}\nonumber\label{e4.3}
		\mc D =& 2_\mu^* \I\Om \I{\eta(w)}^w\left(\I\Om \frac{(w+\hat{u})^{2_\mu^*}+(\eta(w)+\hat{u})^{2_\mu^*}}{|x-y|^\mu}dy\right)\left((t+\hat{u})^{2_\mu^*-1}-(\eta(w)+\hat{u})^{2_\mu^*-1}\right)dtdx\\
		&+2_\mu^* \I\Om \I{\eta(w)}^w\left(\I\Om \frac{(w+\hat{u})^{2_\mu^*}-(\eta(w)+\hat{u})^{2_\mu^*}}{|x-y|^\mu}dy\right)(\eta(w)+\hat{u})^{2_\mu^*-1}dtdx.
	\end{align}
Using the mean value theorem, there exists $\theta \in [0,1]$ such that
\begin{align*}
	\frac{(u+\hat{u})^{2_\mu^*-1}-(w+\hat{u})^{2_\mu^\ast-1}}{u-w}=&(2_\mu^*-1)(u+\hat{u}+\theta(w-u))^{2_\mu^*-2}
	=(2_\mu^*-1)(\hat{u}+(1-\theta)u+\theta w))^{2_\mu^*-2}	\\
	\leq &(2_\mu^*-1)2^{2_\mu^*-3}\left(\hat{u}^{2_\mu^*-2}+((1-\theta)u+\theta w)^{2_\mu^*-2}\right)\\
	\leq& (2_\mu^*-1)2^{2_\mu^*-3}\left(\hat{u}^{2_\mu^*-2}+\max\{u,w\}^{2_\mu^*-2}\right).
\end{align*}
For each $x \in \Om$ and $w \in D(\Phi_{K_{v_2}})$ define the functions
\begin{align*}
	m_w^1(x)=(2_\mu^*-1)2^{2_\mu^*-3}\left(\hat{u}^{2_\mu^*-2}+\max\{|v_2|,|w|\}^{2_\mu^*-2}\right)\chi_{\{w>v_2\}},\\
	\text{and}~m_w^2(x)=2_\mu^*2^{2_\mu^*-2}\left(\hat{u}^{2_\mu^*-1}+\max\{|v_2|,|w|\}^{2_\mu^*-1}\right)\chi_{\{w>v_2\}}.
\end{align*}
Now employing Hardy-Littlewood-Sobolev inequality, we have
\begin{align}\nonumber\label{e4.4}
	\I\Om \I{\eta(w)}^w&\left(\I\Om \frac{(w+\hat{u})^{2_\mu^*}+(\eta(w)+\hat{u})^{2_\mu^*}}{|x-y|^\mu}dy\right)\left((t+\hat{u})^{2_\mu^*-1}-(\eta(w)+\hat{u})^{2_\mu^*-1}\right)dtdx\\ \nonumber
	& \leq \frac12 \I\Om\I\Om \frac{\left((w+\hat{u})^{2_\mu^*}+(\eta(w)+\hat{u})^{2_\mu^*}\right)m_w^1(x)(w-\eta(w))^2}{|x-y|^\mu}dydx\\
	&\leq c_1 \left(|w+\hat{u}|_{2^\ast}^{2_\mu^\ast}+|\eta(w)+\hat{u}|_{2^\ast}^{2_\mu^\ast}\right)|m_w^1(x)(w-\eta(w))^2|_\frac{2^*}{2_\mu^*}
\end{align}
for some appropriate positive constant $c_1$. Similarly with the help of Hardy-Littlewood-Sobolev inequality, H\"{o}lder's inequality, and the definition of $S$ we have
\begin{align}\nonumber\label{e4.5}
	\I\Om \I{\eta(w)}^w\left(\I\Om \frac{(w+\hat{u})^{2_\mu^*}-(\eta(w)+\hat{u})^{2_\mu^*}}{|x-y|^\mu}dy\right)(\eta(w)+\hat{u})^{2_\mu^*-1}dtdx\\ 
	\leq c_2S^\frac{-1}{2}|m_w^2(x)(w-\eta(w))|_\frac{2^*}{2_\mu^*}|\eta(w)+\hat{u}|_{2^*}^{2_\mu^*-1}\|w-\eta(w)\|
\end{align}
for some appropriate positive constant $c_2$. Substituting \eqref{e4.4} and \eqref{e4.5} in \eqref{e4.3}, we obtain
\begin{align}\nonumber\label{e4.6}
	\mc D \leq & c_1 \left(|w+\hat{u}|_{2^\ast}^{2_\mu^\ast}+|\eta(w)+\hat{u}|_{2^\ast}^{2_\mu^\ast}\right)|m_w^1(x)(w-\eta(w))^2|_\frac{2^*}{2_\mu^*} \\
	&+c_2S^\frac{-1}{2}|m_w^2(x)(w-\eta(w))|_\frac{2^*}{2_\mu^*}|\eta(w)+\hat{u}|_{2^*}^{2_\mu^*-1}\|w-\eta(w)\|.
	\end{align}
Suppose on the contrary that the result does not hold. Then there exists a sequence $\{w_k\}_{k \in \mb N}\subset X_0$ such that  $w_k \in K_{w_1}$
and 
\begin{equation*}
	\|w_k-u\| <\frac{1}{2^k},~\Phi_{K_{v_1}}(w_k) <\Phi_{K_{v_1}}(u)~\text{for all}~k\in\mb N.
\end{equation*}
Next we set $ j := u+ \sum_{k=1}^{\infty}|w_k -u|$. Then, clearly $w_k$ satisfies $|w_k| \leq j$ a.e for all $k$. Now for each $ w \in D(\Phi_{K_{v_1}})$, set 
\begin{align*}
	\hat{m}_{w}^1= (2_\mu^*-1)2^{2_\mu^*-3}\left(\hat{u}^{2_\mu^*-2}+\max\{|v_2|,|j|\}^{2_\mu^*-2}\right)\chi_{\{w>v_2\}},\\
	\text{and}~\hat{m}_w^2(x)=2_\mu^*2^{2_\mu^*-2}\left(\hat{u}^{2_\mu^*-1}+\max\{|v_2|,|j|\}^{2_\mu^*-1}\right)\chi_{\{w>v_2\}}.
\end{align*}
Using \eqref{e4.2} and \eqref{e4.6}, we obtain
\begin{align}\nonumber\nonumber\label{e4.7}
	0 >& \Phi_{K_{v_1}}(w_k)-\Phi_{K_{v_1}}(u)\\ \nonumber
	\geq & \Phi_{K_{v_1}}(w_k)-\Phi_{K_{v_1}}(\eta(w_k))\\\nonumber
	\geq & \frac{\|w_k -\eta(w_k)\|^2}{2}-\la\left( c_1 \left(|w_k+\hat{u}|_{2^\ast}^{2_\mu^\ast}+|\eta(w_k)+\hat{u}|_{2^\ast}^{2_\mu^\ast}\right)|\hat{m}_{w_k}^1(x)(w_k-\eta(w_k))^2|_\frac{2^*}{2_\mu^*}\right. \\ \nonumber
	&\left.+c_2S^\frac{-1}{2}|\hat{m}_{w_k}^2(x)(w_k-\eta(w_k))|_\frac{2^*}{2_\mu^*}|\eta(w_k)+\hat{u}|_{2^*}^{2_\mu^*-1}\|w_k-\eta(w_k)\|\right)+\mc J(w_k -\eta(w_k))\\ \nonumber
	\geq & \frac{\|w_k -\eta(w_k)\|^2}{2}+\mc J(w_k -\eta(w_k))-\left( \frac{C_1}{4}|\hat{m}_{w_k}^1(x)(w_k -\eta(w_k))^2|_\frac{2^*}{2_\mu^*}\right.\\
	&+ \left.\frac{C_2}{4}|\hat{m}_{w_k}^2(x)(w_k-\eta(w_k))|_\frac{2^*}{2_\mu^*}\|w_k- \eta(w_k)\|\right),
\end{align}
where $C_1 =\sup_k 4\la c_1 \left(|w_k+\hat{u}|_{2^\ast}^{2_\mu^\ast}+|\eta(w_k)+\hat{u}|_{2^\ast}^{2_\mu^\ast}\right)$ and $C_2 =\sup_k 4\la c_2 S^\frac{-1}{2} |\eta(w_k)+\hat{u}|_{2^*}^{2_\mu^*-1}$. Consider
\begin{align*}
	|\hat{m}_{w_k}^1(x)(w_k -\eta(w_k))^2|_\frac{2^*}{2_\mu^*} \leq &	|\hat{m}_{w_k}^1(x)|_\frac{2^*}{2_\mu^*-2}|(w_k -\eta(w_k))^2|_\frac{22^*}{2_\mu^*}^2\\
	= & \left(\left(\I{\{\hat{m}_{w_k}^1\leq R_1\}}|\hat{m}_{w_k}^1(x)|^\frac{2^*}{2_\mu^*-2}\right)^\frac{2_\mu^*-2}{2^*}+   \left(\I{\{\hat{m}_{w_k}^1> R_1\}}|\hat{m}_{w_k}^1(x)|^\frac{2^*}{2_\mu^*-2}\right)^\frac{2_\mu^*-2}{2^*}\right)\\
	&|(w_k -\eta(w_k))^2|_\frac{22^*}{2_\mu^*}^2.
\end{align*}
Choose $R_1$, $R_2>0$ such that, for all $k$,
\begin{align*}
	C_1S^{-1} \left(\I{\{\hat{m}_{w_k}^1> R_1\}}|\hat{m}_{w_k}^1(x)|^\frac{2^*}{2_\mu^*-2}\right)^\frac{2_\mu^*-2}{2^*} <\frac12,~
	C_2S^\frac{-1}{2}\left(\I{\{\hat{m}_{w_k}^2> R_2\}}|\hat{m}_{w_k}^1(x)|^\frac{2^*}{2_\mu^*-1}\right)^\frac{2_\mu^*-1}{2^*} <\frac12.
\end{align*}
Then, using H\"{o}lder's inequality in \eqref{e4.7} with above estimates, we get
\begin{align*}
	0>& \frac{\|w_k -\eta(w_k)\|^2}{4}+\mc J(w_k -\eta(w_k))-\left( \frac{C_1 R_1}{4}|(w_k-\eta(w_k))^2|^2_\frac{22^*}{2_\mu^*}\right.\\
	&+ \left.\frac{C_2R_2}{4}|(w_k-\eta(w_k))|_\frac{22^*}{2_\mu^*}\|w_k- \eta(w_k)\|\right)\\
	\geq& \frac{\|(w_k -v_2)^+\|^2}{4}+\mc J((w_k -v_2)^+)-\left( \frac{C_1 R_1}{4}|(w_k -v_2)^+|^2_\frac{22^*}{2_\mu^*}\right.\\
	&+ \left.\frac{C_2R_2}{4}|(w_k -v_2)^+|_\frac{22^*}{2_\mu^*}\|w_k- \eta(w_k)\|\right).
\end{align*}
Let $C^* =\max\{\frac{C_1R_1}{2}, \frac{C_2R_2}{2}\}$. Then
\begin{align}\nonumber \label{e4.8}
	0>&\frac{\|(w_k -v_2)^+\|^2}{4}+\mc J((w_k -v_2)^+)\\
	&-\frac{C^*}{2}\left(|(w_k -v_2)^+|^2_\frac{22^*}{2_\mu^*}
	+ |(w_k -v_2)^+|_\frac{22^*}{2_\mu^*}\|(w_k-v_2)^+\|\right).
\end{align}
 Now if $|(w_k-v_2)^+|_\frac{22^*}{2}^2 \leq \frac{1}{L C^*} \|(w_k-v_2)^+\|^2$, then from \eqref{e4.8} we have
\begin{align*}
	0 >\|(w_k-v_2)^+\|^2 \left(\frac{1}{4}-\frac{1}{2L}-\frac{\sqrt{C^*}}{2\sqrt{L}}\right),
\end{align*}
which is contradiction for $L$ large. Now consider the case when  $|(w_k-v_2)^+|_\frac{22^*}{2}^2 \geq \frac{1}{L C^*} \|(w_k-v_2)^+\|^2$.  In this case,
let $\kappa =\inf \{ \mc J (\varphi ): \varphi \in \mc A\}$, where $\mc A =\{ \varphi \in X_0: \varphi \geq 0,|\varphi|_\frac{22^*}{2_\mu^*}=1,~\|\varphi\| \leq  {\sqrt{LC^*}}\}$. Clearly, $\mc A$ is a weakly sequentially closed subset of $X_0$. Also using Fatou's Lemma and the fact that the Riesz potential is a bounded linear functional, one can easily prove that $\mc J$ is a weakly lower semicontinuous on $\mc A$. So, if $\{w_k\}_{k \in \mb N} \subset \mc A$ is a minimizing sequence for $\kappa$ such that $w_k \rightharpoonup w$ as $ k \ra \infty$, then
\begin{align*}
	\mc J(w) \leq \liminf \mc J(w_k).
\end{align*} 
Since $v_2$ is a supersolution of $(\hat{P}_\la)$, $\mc J(w)>0$ for all $w  \in \mc A$. This implies $\kappa >0$. Now notice using the definition of $\kappa$, \eqref{e4.8} can be rewritten as the following:
\begin{align}\nonumber\label{0.2}
	0>&\kappa -\frac{C^*}{2}\left(|(w_k-v_2)^+|_\frac{22^*}{2_\mu^*}+\|(w_k-v_2)^+\|\right)\\
	\geq & \kappa -\frac{C^*}{2} \left( 1+\sqrt{LC^*}\right)|(w_k-v_2)^+|_\frac{22^*}{2_\mu^*}.
\end{align}
As $\{w_k\}_{k \in \mb N}$ is a sequence such that $w_k \ra u$ in $X_0$, it implies that  $|(w_k-v_2)^+|_\frac{22^*}{2_\mu^*} \ra 0$ as $ k \ra \infty$. So from \eqref{0.2}, we get a contradiction to the fact that $\kappa >0$. This completes the proof. \QED
\end{proof}	
The existence of weak solutions to $(P_\lambda)$ follows from the next lemma together with lemma \ref{exi}.
\begin{Lemma}\label{l4.7}
	We have $\La >0$.
\end{Lemma}	
\begin{proof}
	We will use the sub- and supersolution method to prove the required result. From Lemma \ref{l4.4}, we get that $0$ and $\xi$ are the sub- and supersolution, respectively, to $(\hat{P}_\la)$ for $\la$ small enough. We define the closed convex subset of $X_0$ as $K =\{\varphi \in X_0 : 0\leq \varphi \leq \xi\}$. Using the definition fo $K$, we can easily prove that
	\begin{equation*}
		\Phi_K(u) \geq \frac{\|u\|^2}{2}-c_1-c_2
	\end{equation*}
for appropriate positive constants $c_1$ and $c_2$. This imply that $\Phi_K$ is coercive on $K$. Next, we claim that $\Phi_K$ is weakly lower semicontinuous on $K$. Indeed, let $\{\varphi_k\}_{k \in \mb N} \subset K$ such that $\varphi_k \rightharpoonup \varphi$ weakly in $X_0$ as $k \ra \infty$. For each $k$, we have
\begin{align*}
	\I\Om F(x,\varphi_k) dx \leq \I\Om F(x,\varphi) <+\infty,
\end{align*}
\begin{align*}
	\I\Om\I\Om \frac{(\varphi_k+\hat{u})^{2_\mu^*}(\varphi_k+\hat{u})^{2_\mu^*}}{|x-y|^\mu}dxdy \leq 	\I\Om\I\Om \frac{(\xi+\hat{u})^{2_\mu^*}(\xi+\hat{u})^{2_\mu^*}}{|x-y|^\mu}dxdy<+\infty.
\end{align*}
Thus by the dominated convergence theorem and the weak lower semicontonuity of the norm, we deduce that $\Phi_K$ is weakly lower semicontinuous on $K$. Thus, there exists $u \in X_0$ such that
\begin{align*}
	\inf_{\varphi \in K} \Phi_K(\varphi)= \Phi_K(u).
\end{align*}
Since $0 \in \partial^- \Phi_K(u)$, $u$ is a weak solution to $(\hat{P}_\la)$. It implies that $\La>0$. \QED
\end{proof}	
\begin{Theorem}\label{t4.8}
	Let $ \la \in (0,\La)$. Then there exists a positive weak solution $u_\la$ to $(\hat{P}_\la)$ belonging to $X_0$ such that $\Phi(u_\la)<0$ and $u_\la$ is a local minimizer of $\Phi_{K_0}$.
	\end{Theorem}
\begin{proof}
	Let $\la \in (0,\La)$ and $\la_1 \in (\la,\La)$. Then by Lemma \ref{l4.4}, $0$ and $u_{\la_1}$ are strict subsolution and supersolution of $(\hat{P}_\la)$ respectively. The existence of $u_{\la_1}$ is clear by definition of $\La$. Now consider the convex set $K= \{ u \in X_0: 0\leq u \leq u_{\la_1}\}$. Then following the analysis carried out in Lemma \ref{l4.7}, we obtain $u_\la \in X_0$ such that $\inf_{\varphi \in K}\Phi_K(\varphi) =\Phi_K(u_\la)$. Since $0 \in K$ and $\Phi_K(0)<0$, we conclude that $\Phi_K(u_\la)<0$. Let $ v_1 =0$ and $v_2 =u_{\la_1}$ in Theorem \ref{t4.6}, we have $u_\la$ is a local minimizer of $\Phi_{K_0}$. \QED
	\end{proof}	
\begin{Lemma}
	$\La < \infty$.
\end{Lemma}
\begin{proof}
	Suppose on the contrary that $\La =+\infty$. This means that there exist sequences $\{\la_k\}_{k \in \mb N}$ and $\{u_{\la_k}\}_{k \in \mb N}$ such that $\la_k \ra +\infty$ as $k \ra \infty$ and $u_{\la_k}$ is the corresponding solution to $(\hat{P}_{\la_k})$. Then by Theorem \ref{t4.8}, $\Phi(u_{\la_k}) <0$ and $u_{\la_k}$ is a local minimizer of $\Phi_{K_0}$. Thus we have
	\begin{equation}\label{e4.10}
		\frac{1}{2}\|u_{\la_k}\|^2 +\I\Om F(x,u_{\la_k})dx -\frac{\la_k}{2 2_\mu^*}	\I\Om\I\Om \frac{(u_{\la_k}+\hat{u})^{2_\mu^*}(u_{\la_k}+\hat{u})^{2_\mu^*}}{|x-y|^\mu}dxdy <0,
	\end{equation}
and 
	\begin{equation}\label{e4.11}
	\|u_{\la_k}\|^2 +\I\Om f(x,u_{\la_k})u_{\la_k}dx -\la_k	\I\Om\I\Om \frac{(u_{\la_k}+\hat{u})^{2_\mu^*}(u_{\la_k}+\hat{u})^{2_\mu^*-1}u_{\la_k}}{|x-y|^\mu}dxdy =0.
\end{equation}
From \eqref{e4.10} and \eqref{e4.11}, we get
\begin{align*}
0<\frac{\la_k}{2}&\I\Om\I\Om\left( \frac{(u_{\la_k}+\hat{u})^{2_\mu^*}(u_{\la_k}+\hat{u})^{2_\mu^*}-1/2_\mu^*(u_{\la_k}+\hat{u})^{2_\mu^*}(u_{\la_k}+\hat{u})^{2_\mu^*-1}u_{\la_k}}{|x-y|^\mu}\right) dxdy\\
&-\I\Om \left(F(x,u_{\la_k})- \frac12 f(x,u_{\la_k}u_{\la_k})\right)dx.
\end{align*}
By Lemma \ref{l4.3}, we have $F(x,u_{\la_k})-f(x,u_{\la_k})u_{\la_k}/2 \geq 0$ which implies
\begin{align}\label{e4.12}
	\frac12 \I\Om\I\Om \frac{(u_{\la_k}+\hat{u})^{2_\mu^*}(u_{\la_k}+\hat{u})^{2_\mu^*-1}u_{\la_k}}{|x-y|^\mu}dxdy <\frac{1}{2 2_\mu^* }\I\Om\I\Om \frac{(u_{\la_k}+\hat{u})^{2_\mu^*}(u_{\la_k}+\hat{u})^{2_\mu^*}}{|x-y|^\mu}dxdy.
\end{align}
Employing the fact that $\hat{u} \in L^\infty(\Om)$, we conclude the following
\begin{equation*}
\lim_{t \ra \infty}	\frac{\left(\I\Om \frac{(t+\hat{u})^{2_\mu^*}}{|x-y|^\mu}dy\right)(t+\hat{u})^{2_\mu^*}}{\left(\I\Om \frac{(t+\hat{u})^{2_\mu^*}}{|x-y|^\mu}dy\right)(t+\hat{u})^{2_\mu^*-1}t}=1. 
\end{equation*}
Therefore, it follows that for any $\e >0$ small enough, there exists $m_\e >0$ such that, for all $k$
\begin{align}\nonumber \label{e4.13}
	\frac{1}{2 2_\mu^* }\I\Om\I\Om \frac{(u_{\la_k}+\hat{u})^{2_\mu^*}(u_{\la_k}+\hat{u})^{2_\mu^*}}{|x-y|^\mu}&dxdy\\
	 &< \frac{1}{2+\e}\I\Om\I\Om \frac{(u_{\la_k}+\hat{u})^{2_\mu^*}(u_{\la_k}+\hat{u})^{2_\mu^*-1}u_{\la_k}}{|x-y|^\mu}dxdy +m_\e.
\end{align}
Combining \eqref{e4.12} and \eqref{e4.13}, we see that
\begin{equation*}
	\I\Om\I\Om \frac{(u_{\la_k}+\hat{u})^{2_\mu^*}(u_{\la_k}+\hat{u})^{2_\mu^*-1}u_{\la_k}}{|x-y|^\mu}dxdy <\infty~\text{for all}~k.
\end{equation*}
Now from \eqref{e4.11}, we have
\begin{equation*}
	\|u_{\la_k}\|^2 < \la_k	\I\Om\I\Om \frac{(u_{\la_k}+\hat{u})^{2_\mu^*}(u_{\la_k}+\hat{u})^{2_\mu^*-1}u_{\la_k}}{|x-y|^\mu}dxdy.
\end{equation*}
This means $\{\la_k^{-1/2}u_{\la_k}\}_{k \in \mb N}$ is uniformly bounded in $X_0$. Then there exists $w_0 \in X_0$ such that $w_k :=\la_k^{-1/2}u_{\la_k} \rightharpoonup w_0$ weakly in $X_0$. Let $ 0\leq \psi \in C_c^\infty(\Om)$ be a nontrivial function. Let $m>0$ such that $\hat{u}>m$ on supp$(\psi)$. Again using \eqref{e4.11}, we deduce that
\begin{align*}
	\sqrt{\la_k} \I\Om\I\Om\frac{m^{22_\mu^*-1}\psi}{|x-y|^\mu}dxdy \leq& \sqrt{\la_k}\I\Om\I\Om \frac{(u_{\la_k}+\hat{u})^{2_\mu^*}(u_{\la_k}+\hat{u})^{2_\mu^*-1}\psi}{|x-y|^\mu}dxdy\\
	=& \ld w_k, \psi\rd +\frac{1}{\sqrt{\la_k}}\I\Om f(x,u_{\la_k})\psi dx\\
	\leq& \ld w_k, \psi\rd +\frac{1}{\sqrt{\la_k}}\I\Om k^{-\gamma}\psi dx.
\end{align*}
Now passing the limit $k \ra \infty$, we have $\ld w_0,\psi\rd =\infty$, which is not true. Hence $\La <\infty$. \QED 
	\end{proof}	
	\subsection{Second solution} 
In this subsection we will prove the existence of a second solution to $(\hat{P}_\la)$. We denote by $v$ the first solution to $(\hat{P}_\la)$ as obtained in Theorem \ref{t4.8}.
\begin{Proposition}\label{p4.10}
	The functional $\Phi_{K_{v}}$ satisfies the $(CPS)_d$ for each $d$ satisfying
	\begin{equation*}
		d < \Phi_{K_{v}}(v) +\frac12 \left(\frac{n-\mu+2}{2n-\mu}\right) \left(\frac{S_{H,L}^\frac{2n-\mu}{n-\mu+2}}{\la^\frac{n-2}{n-\mu+2}}\right).
	\end{equation*}
\end{Proposition}
\begin{proof}
	Let $ 	d < \Phi_{K_{v}}(v) +\ds\frac12 \left(\frac{n-\mu+2}{2n-\mu}\right) \left(\frac{S_{H,L}^\frac{2n-\mu}{n-\mu+2}}{\la^\frac{n-2}{n-\mu+2}}\right)$ be fixed and choose a sequence $\{v_k\}_{k \in \mb N} \subset D(\Phi_{K_v})$ such that
	\begin{equation*}
		\Phi_{K_v} (v_k) \ra d~\text{and}~(1+\|v_k\|)||| \partial^- \Phi_{K_v}(v_k)||| \ra 0~\text{as}~k \ra \infty.
	\end{equation*}
It implies there exists $ \al_k \in \partial^- \Phi_{K_v}(v_k)$ such that $\|\al_k\| =||| \partial^- \Phi_{K_v}(v_k)|||$ for every $k$. Using Lemma \ref{l4.1}, for each $w \in D(\Phi_{K_v})$ and for each $k$, $f(\cdot,v_k)(w -v_k) \in L^1(\Om)$ and
\begin{align}\nonumber\label{e4.14}
	\ld \al_k,w-v_k\rd \leq \ld v_k, w-v_k\rd &+\I\Om f(x,v_k)(w-v_k)dx\\
	&-\la \I\Om\I\Om\frac{(v_k+\hat{u})^{2_\mu^*}(v_k+\hat{u})^{2_\mu^*-1}(w-v_k)}{|x-y|^\mu}dxdy.
\end{align}
Using the fact that $F(\cdot,v_k) \in L^1(\Om)$ and Lemma \ref{l4.3}, we obtain that $F(\cdot,2v_k) \in L^1(\Om)$. So $2v_k \in D(\Phi_{K_v})$. Taking $w =2v_k$ in \eqref{e4.14}, we get
\begin{equation*}
		\ld \al_k,v_k\rd \leq \| v_k\|^2 +\I\Om f(x,v_k)v_kdx
	-\la \I\Om\I\Om\frac{(v_k+\hat{u})^{2_\mu^*}(v_k+\hat{u})^{2_\mu^*-1}v_k}{|x-y|^\mu}dxdy.
\end{equation*}
Now using Lemma \ref{l4.3}, \eqref{e4.13}, and \eqref{e4.14}, for $\e >0$ small enough, we have
\begin{align*}
	d+1 \geq& \frac12 \|v_k\|^2 +\I\Om F(x,v_k)dx -\frac{\la}{2 2_\mu^* }\I\Om\I\Om \frac{(v_k+\hat{u})^{2_\mu^*}(v_k+\hat{u})^{2_\mu^*}}{|x-y|^\mu}dxdy\\
	\geq & \frac12 \|v_k\|^2 +\I\Om F(x,v_k)dx +\frac{1}{2+\e}\left(\ld \al_k,v_k\rd -\|v_k\|^2 -\I\Om f(x,v_k)v_kdx\right)-\la m_\e\\
	\geq & \frac12 \|v_k\|^2  +\frac{1}{2+\e}\left(\ld \al_k,v_k\rd -\|v_k\|^2 \right)-\la m_\e.
\end{align*}
It shows that $\{v_k\}_{k \in \mb N}$ is a bounded sequence in $X_0$. Hence, up to a subsequence, there exists $v_0 \in X_0$ such that $v_k \rightharpoonup v_0$ weakly in $X_0$ as $k \ra \infty$. We assume, again up to a subsequence, that as $k \ra \infty$,
\begin{align*}
	\|v_k-v_0\|^2 \ra a^2~\text{and}~\I\Om\I\Om \frac{(v_k-v_0)^{2_\mu^*}(v_k-v_0)^{2_\mu^*}}{|x-y|^\mu}dxdy \ra b^{22_\mu^*}~\text{ as}~ k \ra \infty.
\end{align*}
Using the convexity of the function $F$, Brezis-Lieb Lemma and \eqref{e4.14}, we deduce that 
\begin{align*}
	\I\Om F(x,v_0)dx \geq& \I\Om F(x,v_k)dx +\I\Om f(x,v_k)(v_0-v_k)dx\\
	\geq& \I\Om F(x,v_k)dx  -\la \I\Om\I\Om\frac{(v_k+\hat{u})^{2_\mu^*}(v_k+\hat{u})^{2_\mu^*-1}(v_k-v_0)}{|x-y|^\mu}dxdy-\ld \al_k, v_k-v_0\rd\\
	&+\ld v_k, v_k -v_0\rd\\
	=&\I\Om F(x,v_k)dx-\ld \al_k, v_k-v_0\rd
	+\la \I\Om\I\Om\frac{(v_k+\hat{u})^{2_\mu^*}(v_k+\hat{u})^{2_\mu^*-1}(v_0+\hat{u})}{|x-y|^\mu}dxdy\\
	&-\la \I\Om\I\Om\frac{(v_0+\hat{u})^{2_\mu^*}(v_0+\hat{u})^{2_\mu^*}+(v_k-v_0)^{2_\mu^*}(v_k-v_0)^{2_\mu^*}}{|x-y|^\mu}dxdy+\ld v_k, v_k -v_0\rd.
\end{align*}
Taking into account the weak convergence of $v_k \rightharpoonup v_0$ in $X_0$, we obtain as $k \ra \infty$
\begin{equation*}
	\I\Om F(x,v_0) dx \geq \I\Om F(x,v_0) +a^2 -\la b^{22_\mu^*}.
\end{equation*}
This implies
\begin{equation}\label{e4.15}
	\la b^{2 2_\mu^*} \geq a^2.
\end{equation}
Now since $v$ is a weak positive solution to $(\hat{P}_\la)$, for each $k$, we have
\begin{align}\label{e4.16}
	0= \ld v, v_k-v\rd +\I\Om f(x,v)(v_k-v)dx +\la \I\Om\I\Om\frac{(v+\hat{u})^{2_\mu^*}(v+\hat{u})^{2_\mu^*-1}(v_k-v)}{|x-y|^\mu}dxdy.
\end{align}
Noting that $F(\cdot,v_k),~F(\cdot,2v_k) \in L^1(\Om)$ and $v \leq 2v_k -v \leq 2 v_k$, we infer that $2v_k -v \in D(\Phi_{K_v})$. Testing  \eqref{e4.14} with $2v_k -v$, we obtain
\begin{align}\nonumber \label{e4.17}
		\ld \al_k,v_k-v\rd \leq& \ld v_k,v_k-v\rd +\I\Om f(x,v_k)(v_k-v)dx\\
	&-\la \I\Om\I\Om\frac{(v_k+\hat{u})^{2_\mu^*}(v_k+\hat{u})^{2_\mu^*-1}(v_k-v)}{|x-y|^\mu}dxdy.
\end{align}
Taking into account Lemma \ref{l4.3}, \eqref{e4.16} and \eqref{e4.17}, we have
\begin{align}\nonumber\label{e4.18}
	\Phi_{K_v}(v_k)-\Phi_{K_v}(v)=&\frac12 \|v_k\|^2 +\I\Om F(x,v_k)dx -\frac{\la}{2 2_\mu^*}\I\Om\I\Om\frac{(v_k+\hat{u})^{2_\mu^*}(v_k+\hat{u})^{2_\mu^*}}{|x-y|^\mu}dxdy\\ \nonumber
	&- \frac12 \|v\|^2 -\I\Om F(x,v)dx +\frac{\la}{2 2_\mu^*}\I\Om\I\Om\frac{(v+\hat{u})^{2_\mu^*}(v+\hat{u})^{2_\mu^*}}{|x-y|^\mu}dxdy\\ \nonumber
	\geq & \I\Om \left(F(x,v_k)-F(x,v)-\frac12 (f(x,v)+f(x,v_k))(v_k-v)\right)dx\\ \nonumber
	&+ \frac{\la}{2 2_\mu^*}\I\Om\I\Om\frac{(v+\hat{u})^{2_\mu^*}(v+\hat{u})^{2_\mu^*}-(v_k+\hat{u})^{2_\mu^*}(v_k+\hat{u})^{2_\mu^*}}{|x-y|^\mu}dxdy +\frac12 \ld \al_k,v_k-v\rd\\ \nonumber
	&+\frac{\la}{2}\I\Om\I\Om\frac{\left((v+\hat{u})^{2_\mu^*}(v+\hat{u})^{2_\mu^*-1}-(v_k+\hat{u})^{2_\mu^*}(v_k+\hat{u})^{2_\mu^*-1}\right)(v_k-v  )}{|x-y|^\mu}dxdy\\ \nonumber
		\geq & \frac{\la}{2 2_\mu^*}\I\Om\I\Om\frac{(v+\hat{u})^{2_\mu^*}(v+\hat{u})^{2_\mu^*}-(v_k+\hat{u})^{2_\mu^*}(v_k+\hat{u})^{2_\mu^*}}{|x-y|^\mu}dxdy +\frac12 \ld \al_k,v_k-v\rd\\ \nonumber
		&+\frac{\la}{2}\I\Om\I\Om\frac{(v+\hat{u})^{2_\mu^*}(v+\hat{u})^{2_\mu^*-1}(v_k-v)-(v_k+\hat{u})^{2_\mu^*}(v_k+\hat{u})^{2_\mu^*-1}(v +\hat{u}  )}{|x-y|^\mu}dxdy \\  &+\frac{\la}{2}\I\Om\I\Om \frac{(v_k+\hat{u})^{2_\mu^*}(v_k+\hat{u})^{2_\mu^*}}{|x-y|^\mu}dxdy =: \mc P +\frac12 \ld \al_k,v_k-v\rd. 
\end{align}
Again using Brezis-Lieb Lemma, we have
\begin{align} \nonumber\label{e4.19}
	\mc P=& \la \left(\frac12 -\frac{1}{2 2_\mu^*}\right)\I\Om\I\Om \frac{(v_k-v_0)^{2_\mu^*}(v_k-v_0)^{2_\mu^*}+({v_0}+\hat{u})^{2_\mu^*}(v_0+\hat{u})^{2_\mu^*}}{|x-y|^\mu}dxdy \\ \nonumber
	&+\frac{\la}{2}\I\Om\I\Om\frac{(v+\hat{u})^{2_\mu^*}(v+\hat{u})^{2_\mu^*-1}(v_k-v)-(v_k+\hat{u})^{2_\mu^*}(v_k+\hat{u})^{2_\mu^*-1}(v +\hat{u}  )}{|x-y|^\mu}dxdy\\
	&+\frac{\la}{22_\mu^*}\I\Om\I\Om \frac{(v+\hat{u})^{2_\mu^*}(v+\hat{u})^{2_\mu^*}}{|x-y|^\mu}dxdy+o(1).
\end{align}
Using the weak convergence of the sequence $\{v_k\}_{k \in \mb N}$, we have as $k\to\infty$
\begin{align}\label{e4.20}
\noindent	\I\Om\I\Om\frac{(v+\hat{u})^{2_\mu^*}(v+\hat{u})^{2_\mu^*-1}(v_k-v_0)}{|x-y|^\mu}dxdy \ra 0 ~\text{and}
\end{align}
\begin{align}\label{e4.21}
	\I\Om\I\Om\frac{\left((v_k+\hat{u})^{2_\mu^*}(v_k+\hat{u})^{2_\mu^*-1}-(v_0+\hat{u})^{2_\mu^*}(v_0+\hat{u})^{2_\mu^*-1}\right)(v +\hat{u}  )}{|x-y|^\mu}dxdy \ra 0.
\end{align}
Combining \eqref{e4.18}-\eqref{e4.21} and passing the limit $ k \ra  \infty$, we obtain that
	\begin{align}\nonumber\label{e4.22}
	d- \Phi_{K_v}(v) \geq & \frac{\la}{22_\mu^*}\I\Om\I\Om \frac{(v+\hat{u})^{2_\mu^*}(v+\hat{u})^{2_\mu^*}}{|x-y|^\mu}dxdy +\la \left(\frac12 -\frac{1}{2 2_\mu^*}\right) b^{22_\mu^*}\\  \nonumber
	&+ \frac{\la}{2}\I\Om\I\Om\frac{\left((v+\hat{u})^{2_\mu^*}(v+\hat{u})^{2_\mu^*-1}+(v_0+\hat{u})^{2_\mu^*}(v_0+\hat{u})^{2_\mu^*-1}\right)(v_0-v)}{|x-y|^\mu}dxdy\\ 
	&-\frac{\la}{22_\mu^*}\I\Om\I\Om \frac{(v_0+\hat{u})^{2_\mu^*}(v_0+\hat{u})^{2_\mu^*}}{|x-y|^\mu}dxdy := \mc P_1(\text{say}) + \la \left(\frac12 -\frac{1}{2 2_\mu^*}\right) b^{22_\mu^*}.
\end{align}
Next we will show that $\mc P_1 \geq 0$. Indeed, we have
	\begin{align}	\nonumber\label{e4.23}
	\mc P_1 =&\frac{\la}{2}\I\Om\I\Om\frac{(v+\hat{u})^{2_\mu^*}\left((v+\hat{u})^{2_\mu^*-1}+(v_0+\hat{u})^{2_\mu^*-1}\right)(v_0-v)}{|x-y|^\mu}dxdy\\ 	\nonumber
	&-\frac{\la}{2}\I\Om \I\Om \frac{(v+\hat{u})^{2_\mu^*}(v_0+\hat{u})^{2_\mu^*-1}(v_0-v)}{|x-y|^\mu}dxdy\\ 	\nonumber
	&+\frac{\la}{2}\I\Om\I\Om\frac{(v_0+\hat{u})^{2_\mu^*}\left((v+\hat{u})^{2_\mu^*-1}+(v_0+\hat{u})^{2_\mu^*-1}\right)(v_0-v)}{|x-y|^\mu}dxdy\\ 	\nonumber
	&-\frac{\la}{2}\I\Om \I\Om \frac{(v_0+\hat{u})^{2_\mu^*}(v+\hat{u})^{2_\mu^*-1}(v_0-v)}{|x-y|^\mu}dxdy\\ \nonumber
	& + \frac{\la}{22_\mu^*}\I\Om\I\Om \frac{(v+\hat{u})^{2_\mu^*}\left((v+\hat{u})^{2_\mu^*}-(v_0+\hat{u})^{2_\mu^*}\right)}{|x-y|^\mu}dxdy\\
		& + \frac{\la}{22_\mu^*}\I\Om\I\Om \frac{(v_0+\hat{u})^{2_\mu^*}\left((v+\hat{u})^{2_\mu^*}-(v_0+\hat{u})^{2_\mu^*}\right)}{|x-y|^\mu}dxdy.
\end{align}
Since 
\begin{align*}
	(v+\hat{u})^{2_\mu^*} -(v_0 +\hat{u})^{2_\mu^*} =-2_\mu^* \I{v}^{v_0}(t+\hat{u})^{2_\mu^*-1}dt \geq -2_\mu^* \left(\frac{	(v+\hat{u})^{2_\mu^*-1}+(v_0 +\hat{u})^{2_\mu^*-1}}{2}\right)(v_0-v),
\end{align*}
we obtain
\begin{align}\nonumber \label{e4.24}
	\frac{\la}{22_\mu^*}\I\Om\I\Om& \frac{(v+\hat{u})^{2_\mu^*}\left((v+\hat{u})^{2_\mu^*}-(v_0+\hat{u})^{2_\mu^*}\right)}{|x-y|^\mu}dxdy\\
	 \geq& -\frac{\la}{4}\I\Om\I\Om \frac{(v+\hat{u})^{2_\mu^*}\left((v+\hat{u})^{2_\mu^*-1}+(v_0+\hat{u})^{2_\mu^*-1}\right)(v_0-v)}{|x-y|^\mu}dxdy.
\end{align}
Similarly, we have
\begin{align}\nonumber \label{e4.25}
	\frac{\la}{22_\mu^*}\I\Om\I\Om& \frac{(v_0+\hat{u})^{2_\mu^*}\left((v+\hat{u})^{2_\mu^*}-(v_0+\hat{u})^{2_\mu^*}\right)}{|x-y|^\mu}dxdy\\
	\geq& -\frac{\la}{4}\I\Om\I\Om \frac{(v_0+\hat{u})^{2_\mu^*}\left((v+\hat{u})^{2_\mu^*-1}+(v_0+\hat{u})^{2_\mu^*-1}\right)(v_0-v)}{|x-y|^\mu}dxdy.
\end{align}
From \eqref{e4.23}, \eqref{e4.24} and \eqref{e4.25}, we deduce that
\begin{align}\nonumber \label{e4.26}
	\mc P_1 =& \frac{\la}{4}\I\Om\I\Om \frac{(v+\hat{u})^{2_\mu^*}\left((v+\hat{u})^{2_\mu^*-1}-(v_0+\hat{u})^{2_\mu^*-1}\right)(v_0-v)}{|x-y|^\mu}dxdy\\ \nonumber
	&+ \frac{\la}{4}\I\Om\I\Om \frac{(v_0+\hat{u})^{2_\mu^*}\left((v_0+\hat{u})^{2_\mu^*-1}-(v+\hat{u})^{2_\mu^*-1}\right)(v_0-v)}{|x-y|^\mu}dxdy\\ \nonumber
	=&\frac{\la}{4}\I\Om\I\Om \frac{\left((v_0+\hat{u})^{2_\mu^*}-(v+\hat{u})^{2_\mu^*}\right)\left((v_0+\hat{u})^{2_\mu^*-1}-(v+\hat{u})^{2_\mu^*-1}\right)(v_0-v)}{|x-y|^\mu}dxdy\\
	\geq& 0.
\end{align}
Hence from \eqref{e4.22} and \eqref{e4.26}, we get
\begin{equation}\label{e4.27}
	d- \Phi_{K_v}(v) \geq \la \left(\frac12 -\frac{1}{2 2_\mu^*}\right) b^{22_\mu^*}.
\end{equation}
Using definition of $S_{H,L}$ and \eqref{e4.15}, we have $\la b^{22_\mu^*} \geq a^2$ and $a^2 \geq S_{H,L} b^2$, that is
\begin{equation}\label{e4.28}
	b \geq \left( \frac{S_{H,L}}{\la}\right)^{\frac{n-2}{2(n-\mu+2)}}.
\end{equation}
Using \eqref{e4.27} and \eqref{e4.28}, we get
\begin{align*}
		d- \Phi_{K_v}(v) \geq \la \left(\frac12 -\frac{1}{2 2_\mu^*}\right) \left( \frac{S_{H,L}}{\la}\right)^{\frac{2n-\mu}{(n-\mu+2)}} = \frac12 \left(\frac{n-\mu+2}{2n-\mu}\right) \left(\frac{S_{H,L}^\frac{2n-\mu}{n-\mu+2}}{\la^\frac{n-2}{n-\mu+2}}\right).
\end{align*}
It contradicts the fact that $ 	d < \Phi_{K_{v}}(v) +\ds\frac12 \left(\frac{n-\mu+2}{2n-\mu}\right) \left(\frac{S_{H,L}^\frac{2n-\mu}{n-\mu+2}}{\la^\frac{n-2}{n-\mu+2}}\right)$. Hence $a=0$. \QED
	\end{proof}	
Now consider the family of minimizers of the best constant $S_{H,L}$ (see Lemma \ref{l2.2} ) given by
\begin{align*}
	V_\e(x) =S^\frac{(n-\mu)(2-n)}{4(n-\mu+2)} (C(n,\mu))^\frac{2-n}{2(n-\mu+2)}\left(\frac{\e}{\e^2+|x|^2}\right)^\frac{n-2}{2},~0<\e<1.
\end{align*}
	Let $\de >0$ such that $B_{4\de} \subset \Om$. Now define $\psi \in C_c^\infty(\Om)$ such that $0 \leq \psi \leq 1$ in $\mb R^n$, $\psi \equiv 1$ in $B_\de(0)$ and $\psi \equiv 0$ in $\mb R^n \setminus B_{2\de}(0)$. For each $\e >0$ and $x \in \mb R^n$, we define $u_\e(x) =\psi(x)V_\e(x)$. Then we have the following:
	\begin{Proposition}\label{p4.11}
		Let $n \geq 3$, $0<\mu <n$ then the following holds:
		\begin{enumerate}[label =(\roman*)]
			\item  $\ll u_\e \gg^2 \leq S_{H,L}^\frac{2n-\mu}{n-\mu+2}+O (\e^{n-2}).$
			\item $\|u_\e\|_{HL}^{22_\mu^*} \leq S_{H,L}^\frac{2n-\mu}{n-\mu+2} +O(\e^n)$.
			\item $\|u_\e\|_{HL}^{22_\mu^*} \geq  S_{H,L}^\frac{2n-\mu}{n-\mu+2} -O(\e^n)$.
			\item $[u _\e]_s^2 =O(\e^{\nu_{s,n}})$, where $\nu_{s,n}=\min\{n-2,2-2s\}$.
		\end{enumerate}
	\end{Proposition}
\begin{proof}
	For proof of part $(i)$, we refer to \cite[Lemma 1.46]{W}. For $(ii)$ and $(iii)$, see \cite[Proposition 2.8]{GMS1}. Lastly for a proof of part $(iv) $, see \cite[p22]{BDVV}.\QED 
	\end{proof}	
\begin{Lemma}\label{l4.12}
	The following holds:
	\begin{enumerate}[label=(\roman*)]
		\item If $ \mu < \min\{4,n\}$ then for all $\zeta <1$,
		\begin{align*}
			\| v + t u_\e\|_{HL}^{22_\mu^*} \geq& \|v\|_{HL}^{22_\mu^*}+\|u_\e\|_{HL}^{22_\mu^*} +\widetilde{C} t^{22_\mu^* -1}\I\Om\I\Om \frac{(u_\e(x))^{2_\mu^*}(u_\e(y))^{2_\mu^*-1}v(y)}{|x-y|^\mu}dxdy\\
			+& 22_\mu^*t \I\Om\I\Om \frac{(v(x))^{2_\mu^*}(v(y))^{2_\mu^*-1} u_\e(y)}{|x-y|^\mu}dxdy -O(\e^{(\frac{2n-\mu}{4})\zeta}).
				\end{align*}
			\item There exists a constant $T_0 >0$ such that $\ds\I\Om\I\Om \frac{(u_\e(x))^{2_\mu^*}(u_\e(y))^{2_\mu^*-1} v(y)}{|x-y|^\mu}dxdy \geq \ds \widetilde{C} T_0 \e^{\frac{n-2}{2}}$.
	\end{enumerate}
\end{Lemma}
\begin{proof}
	For a proof, see the proof of  \cite[Lemma 4.2]{GS}. \QED
	\end{proof}
	\begin{Lemma}\label{l4.13}
		We have
		\begin{equation*}
			\sup\{\Phi_{K_v}(v+t u_\e):t \geq 0\} < \Phi_{K_v}(v)+
			 	\ds\frac12 \left(\frac{n-\mu+2}{2n-\mu}\right) \left(\frac{S_{H,L}^\frac{2n-\mu}{n-\mu+2}}{\la^\frac{n-2}{n-\mu+2}}\right)
		\end{equation*}
	for any sufficiently small $\e >0$ and $n+2s <6$. 
	\end{Lemma}
\begin{proof}
	Taking into account the fact that $v$ is a weak solution of $({P}_\la)$ and employing Lemma \ref{l4.12}, for all $\zeta <1$, we have
	\begin{align*}
		\Phi_{K_v}(v+ t u_\e) -\Phi_{K_v}(v) \leq& \frac12 \|t u_\e\|^2-\frac{\la}{2 2_\mu^*}\|t u_\e\|_{HL}^{22_\mu^*} 
		+ \I\Om \left(F(v + tu_\e)-F(x,v)-f(x,v)t u_\e\right)dx\\
		&-\frac{\la\widetilde{C}t^{22_\mu^*-1}}{22_\mu^*}\I\Om\I\Om\frac{(u_\e(x))^{2_\mu^*}(u_\e(y))^{2_\mu^*-1}v(y)}{|x-y|^\mu}dxdy+O(\e^{(\frac{2n-\mu}{4})\zeta}).
	\end{align*}
Using Proposition \ref{p4.11} and Lemma \ref{l4.12}, we obtain
\begin{align}\nonumber \label{e4.29}
		\Phi_{K_v}(v+ t u_\e) -\Phi_{K_v}(v) \leq& \frac{t^2}{2}\left(S_{H,L}^\frac{2n-\mu}{n-\mu+2}+O(\e^{\nu_{s,n}})\right)-\frac{\la t^{22_\mu^*}}{2 2_\mu^*} \left(S_{H,L}^\frac{2n-\mu}{n-\mu+2}-O(\e^n)\right)+O(\e^{(\frac{2n-\mu}{4})\zeta})\\
		&+\I\Om \left(F(v + tu_\e)-F(x,v)-f(x,v)t u_\e\right)dx-\frac{\la\widetilde{C}t^{22_\mu^*-1}}{22_\mu^*} \widetilde{C} T_0 \e^\frac{n-2}{2}.
	\end{align}
We see that for any  fix $1<\rho <\min\{2,\frac{2}{n-2}\}$, there exists $T_1>0$ such that
$\ds	\I\Om |u_\e|^\rho dx \leq T_1 \e^\frac{(n-2)\rho}{2}.$ Moreover, there exists $T_2 >0$ such that, for all $x \in \Om$, $p>m$ and $r \geq 0$,
\begin{equation*}
	F(x,p+r)-F(x,r) -f(x,p)r=\I{p}^{p+r} (\tau^{-\gamma}-p^{-\gamma})d\tau \leq T_2 r^\rho.
\end{equation*}
Using last inequality and \eqref{e4.29} with $\zeta =\frac{2}{2_\mu^*}$, we obtain
\begin{align*}
	\Phi_{K_v}(v+ t u_\e) -\Phi_{K_v}(v) \leq& \frac{t^2}{2}\left(S_{H,L}^\frac{2n-\mu}{n-\mu+2}+O(\e^{\nu_{s,n}})\right)-\frac{\la t^{22_\mu^*}}{2 2_\mu^*} \left(S_{H,L}^\frac{2n-\mu}{n-\mu+2}-O(\e^n)\right)\\
	&-\frac{\la\widetilde{C}t^{22_\mu^*-1}}{22_\mu^*} \widetilde{C} T_0 \e^\frac{n-2}{2}+T_1T_2t^\rho \e^\frac{(n-2)\rho}{2} +o\left(\e^\frac{n-2}{2}\right)\\
			:=& g(t).
\end{align*}
Clearly, $ g(t) \ra -\infty$ as $t \ra \infty$, $g(t)>0$ as $t\ra 0^+$ and there exists $ t_\e >0$ such that $g^\prime(t_\e)=0$. Furthermore, there exists positive constants $R_1$ and $R_2$ such that $R_1 \leq t_\e \leq R_2$. Hence
\begin{align*}
	g(t) 
	 \leq& \frac{t_\e^2}{2}\left(S_{H,L}^\frac{2n-\mu}{n-\mu+2}+O(\e^{\nu_{s,n}})\right)-\frac{\la t_\e^{22_\mu^*}}{2 2_\mu^*} \left(S_{H,L}^\frac{2n-\mu}{n-\mu+2}-O(\e^n)\right)\\
	&-\frac{\la\widetilde{C}R_1^{22_\mu^*-1}}{22_\mu^*} \widetilde{C} T_0 \e^\frac{n-2}{2}+T_1T_2R_2^\rho \e^\frac{(n-2)\rho}{2} +o\left(\e^\frac{n-2}{2}\right)\\
	\leq & \sup_{t \geq 0}g_1(t)-\frac{\la\widetilde{C}R_1^{22_\mu^*-1}}{22_\mu^*} \widetilde{C} T_0 \e^\frac{n-2}{2}+T_1T_2R_2^\rho \e^\frac{(n-2)\rho}{2} +o\left(\e^\frac{n-2}{2}\right),
\end{align*}
where $g_1(t)=\ds\frac{t^2}{2}\left(S_{H,L}^\frac{2n-\mu}{n-\mu+2}+O(\e^{\nu_{s,n}})\right)-\frac{\la t^{22_\mu^*}}{2 2_\mu^*} \left(S_{H,L}^\frac{2n-\mu}{n-\mu+2}-O(\e^n)\right)$. On trivial computation, we get
\begin{align*}
	\Phi_{K_v}(v+ t u_\e) -\Phi_{K_v}(v) \leq& 	\ds\frac12 \left(\frac{n-\mu+2}{2n-\mu}\right) \left(\frac{S_{H,L}^\frac{2n-\mu}{n-\mu+2}}{\la^\frac{n-2}{n-\mu+2}}\right)+O(\e^{\nu_{s,n}}) -C \e^\frac{n-2}{2}+o(\e^\frac{n-2}{2})
\end{align*}
for an appropriate constant $C>0$. Thus, for $\e $ sufficiently small and owing to the assumption $n+2s <6$, we obtain 
\begin{align*}
		\Phi_{K_v}(v+ t u_\e) -\Phi_{K_v}(v) \leq& 	\ds\frac12 \left(\frac{n-\mu+2}{2n-\mu}\right) \left(\frac{S_{H,L}^\frac{2n-\mu}{n-\mu+2}}{\la^\frac{n-2}{n-\mu+2}}\right).
\end{align*}
This completes the proof. \QED
	\end{proof}	
\begin{Proposition}\label{p4.14}
	Assuming $n+2s<6
	$, there exists two distinct solutions to $(\hat{P}_\la)$, for any $\la \in (0,\La)$.
\end{Proposition}	
	\begin{proof}
	From Theorem \ref{t4.8}, we have $v$ is a local minimizer of $\Phi_{K_v}$. This imply that there exists $ \zeta >0$ such that $ \Phi_{K_v}(w) \geq \Phi_{K_v}(v)$ for every $w \in K_v$ with $\|w-v\| \leq \zeta$. Let $u =u_\e$ for $\e $ obtained in Lemma \ref{l4.13}. Since $ \Phi_{K_v}(v+t u) \ra -\infty$ as $t \ra \infty$, so choose $t \geq \zeta/\|u\|$ such that $\Phi_{K_v}(v+t u) \leq \Phi_{K_v}(v)$. Now	define
	\begin{align*}
		\Sigma = \{ \Psi \in C([0,1], D(\Phi_{K_v})):\Psi(0)=v,~\Psi(1)=v+tu \},
	\end{align*}
\begin{align*}
	A =\{w \in D(\Phi_{K_v}):\|w-v\| =\al\}~\text{and}~d =\inf_{\Psi \in \Sigma} \sup_{r \in [0,1]} \Phi_{K_v}(\Psi(r)).
\end{align*}
Combining Proposition \ref{p4.10} and Lemma \ref{l4.13}, $\Phi_{K_v}$ satisfies $(CPC)_d$ condition. If $d =\Phi_{K_v}(v) =\inf \Phi_{K_v}(A)$, then $v \not \in A$, $v+t u \not \in A$, $\inf \Phi_{K_v}(A)\geq \Phi_{K_v}(v)\geq \Phi_{K_v}(v+tu)$, and for every $\Psi \in \Sigma$, there exists $r \in [0,1]$ such that $\|\Psi(r)-v\|=\zeta$. Thus by Theorem \ref{t2.13}, we get there exists $ w \in D(\Phi_{K_v})$ such that $w \ne v$, $\Phi_{K_v} (w)=d$ and $0 \in \partial ^- \Phi_{K_v}(w)$. Using Proposition \ref{p4.2}, we obtain that $w$ is a positive weak solution to $(\hat{P}_\la)$. \QED
		\end{proof}
\textbf{End of Proof of Theorem \ref{tmr}:} Combining Lemma \ref{l3.6}, Theorem \ref{t4.8} and Proposition \ref{p4.14}, the proof of Theorem \ref{tmr} is complete. \QED
	
		\textbf{Acknowledgement:} The first author thanks the CSIR(India) for financial support in the form of a Senior Research Fellowship, Grant Number $09/086(1406)/2019$-EMR-I. The second author was partially funded by IFCAM (Indo-French Centre for Applied Mathematics) IRL CNRS 3494.		
		

	\end{document}